\documentclass[12pt]{amsart}
\usepackage{amssymb}
\usepackage[all]{xy}
\usepackage{hyperref}

\usepackage[english, french]{babel}

\newcommand{\SL}{\mathrm{SL}_2}
\newcommand{\SLq}{\mathrm{SL}_2^q}
\newcommand{\Uq}{\mathrm U_q(\mathfrak{sl}_2)}
\newcommand{\QB}[3]{\binom{#1}{#2}_{\kern -2pt#3}}
\newcommand{\QI}[2]{( #1 )_{#2}}
\newcommand{\QQI}[2]{[ #1 ]_{#2}}
\newcommand{\op}{\mathrm{op}}
\newcommand{\Tr}{\mathrm{Trace\,}}
\newcommand{\A}{\mathcal A}
\newcommand{\SSS}{\mathcal S^q}
\newcommand{\C}{\mathbb C}
\newcommand{\Z}{\mathbb Z}
\newcommand{\End}{\mathrm{End}}

\newcommand{\Aut}{\mathrm{Aut}}
\newcommand{\Id}{\mathrm{Id}}

\renewcommand{\leq}{\leqslant}
\renewcommand{\geq}{\geqslant}
\renewcommand{\phi}{\varphi}
\renewcommand{\epsilon}{\varepsilon}

\newtheorem{thm}{Theorem}
\newtheorem{lem}[thm]{Lemma}
\newtheorem{prop}[thm]{Proposition}
\newtheorem{cor}[thm]{Corollary}

\theoremstyle{remark}

\title[Miraculous cancellations for quantum $\SL$]{Miraculous cancellations\\
for  quantum $\SL$}

\author{Francis Bonahon}
\address{Department of Mathematics, University of Southern California, Los Angeles CA 90089-2532, U.S.A.}
\email{fbonahon@math.usc.edu}

\dedicatory{\`A Jean-Pierre Otal,\\
en l'honneur de  ses $3\times4\times5$ ans}

\thanks{This work was partially supported by the grants DMS-1406559 and DMS-1711297 from the U.S. National Science Foundation.} 

\begin{document}
\maketitle

\begin{footnotesize}\leftskip .5in \rightskip .5in
\selectlanguage{french}
 \noindent \textsc{R\'esum\'e fran\c cais.} Des travaux pr\'ec\'edents de Helen Wong et de l'auteur ont mis en \'evidence, quand le param\`etre quantique  $q= \mathrm e^{2\pi\mathrm i \hbar}$ est une racine de l'unit\'e, des \og annulations miraculeuses\fg\ pour l'application de trace quantique qui relie l'alg\`ebre d'\'echeveaux du crochet de Kauffman \`a l'espace de Teichm\"uller quantique d'une surface. L'article ci-dessous fournit une interpr\'etation plus conceptuelle de ce ph\'enom\`ene, en termes de repr\'esentations du groupe quantique $\mathrm U_q(\mathfrak{sl}_2)$ et de son alg\`ebre de Hopf duale  $\mathrm{SL}_2^q$.

 \medskip
 \selectlanguage{english}
 \noindent \textsc{English abstract.} In earlier work, Helen Wong and the author discovered certain ``miraculous cancellations'' for the quantum trace map connecting the Kauffman bracket skein algebra of a surface to its quantum Teichm\"uller space, occurring when the quantum parameter $q= \mathrm e^{2\pi\mathrm i \hbar}$ is a root of unity. The current paper is devoted to giving a more representation theoretic interpretation of this phenomenon, in terms of the quantum group $\mathrm U_q(\mathfrak{sl}_2)$ and its dual Hopf algebra $\mathrm{SL}_2^q$.

\end{footnotesize}

\bigskip
 \selectlanguage{english}

The equation 
\begin{equation}
\label{eqn:Frobenius}
(X+Y)^n=X^n + Y^n
\end{equation}
is (unfortunately) very familiar to some of our students, who find it convenient to ``simplify'' computations. However, it is also well-known that this relation does hold in some cases, for instance in a ring of characteristic $n$ with $n$ prime, or when the variables $X$ and $Y$ satisfy the $q$-commutativity relation that $YX=qXY$ with $q\in \C$ a primitive $n$--root of unity; see~\S \ref{sect:FrobeniusEqn}. 

The structure of Equation~(\ref{eqn:Frobenius}) can be described by considering the two-variable polynomial $P(X,Y)=X+Y$. Then (\ref{eqn:Frobenius}) states that the polynomial $P(X,Y)^n$, obtained by taking the $n$--th power of $P(X,Y)$, coincides with the polynomial $P(X^n, Y^n)$ obtained by replacing the variables $X$, $Y$ with their powers $X^n$, $Y^n$, respectively. 

Helen Wong and the author discovered similar identities in their study of the Kauffman bracket skein algebra of a surface \cite{BonWon2}. These relations involved a 2--dimensional version of the operation of ``taking the $n$--th power'', through the Chebyshev polynomial $T_n(t) \in \Z[t]$ defined by the property that
$$
\Tr A^n = T_n(\Tr A)
$$
for every 2-by-2 matrix $A\in \SL(\C)$ with determinant 1. A typical consequence of the miraculous cancellations discovered in \cite{BonWon2} is that, when $YX=qXY$ and $q$ is a primitive $n$--root of unity, 
\begin{equation}
\label{eqn:CancellationFirstExample}
T_n(X + Y + X^{-1}) = X^n + Y^n + X^{-n},
\end{equation}
which fits the pattern $T_n \bigl( P(X,Y) \bigr) = P(X^n, Y^n)$ for the polynomial $P(X,Y)= X + XY + X^{-1}$. The arguments of \cite{BonWon2} provide many examples of such polynomials, involving several $q$--commuting variables. 

In \cite{BonWon2} these ``Chebyshev cancellations'' were used to connect, when $q$ is a root of unity, irreducible representations of the Kauffman bracket skein algebra $\SSS(S)$ of a surface $S$ to group homomorphisms $\pi_1(S) \to \SL(\C)$. The skein algebra $\SSS(S)$ is a purely topological object whose elements are represented by framed links in the thickened surface $S\times [0,1]$. It draws its origin from Witten's interpretation \cite{Wi, RT1, RT2} of the Jones polynomial knot invariant within the framework of a topological quantum field theory, and as a consequence it is closely connected to the quantum group $\Uq$. The arguments of \cite{BonWon2} were often developed by trial and error. The purpose of the current article is to put these constructions into a more conceptual framework, where the connection with $\Uq$ and $\SL(\C)$ appears more clearly.  Another goal is to emphasize the representation theoretic nature of these phenomena, with the long term objective of generalizing them to quantum knot invariants and skein algebras based on other quantum groups $\mathrm U_q(\mathfrak g)$, such as the $\mathrm U_q(\mathfrak {sl}_n)$--based HOMFLY polynomial and skein algebra. 

In addition to the fact that quantum groups are still an acquired taste for many mathematicians, including the author, the connection between $\Uq$ and $\SL(\C)$ is more intuitive if we replace $\Uq$ with its dual Hopf algebra $\SLq$, in the sense of \cite{RTF, Man1, Man2, Tak1, Tak2}. This will enable us to express our constructions solely in terms of 2-by-2 matrices with coefficients in an arbitrary noncommutative algebra $\A$; in \cite{BonWon2}, the algebra $\A$ was the quantum Teichm\"uller space of the surface. This point of view is sufficiently close to $\SL(\C)$ that it should be relatively intuitive for those mathematicians who have a long track record in hyperbolic geometry, since $\mathrm{PSL}_2(\C)$ is the isometry group of the hyperbolic space $\mathbb H^3$. This category includes the author and Jean-Pierre Otal, and it is a pleasure to dedicate this article to him as an acknowledgement of the great influence that he had on the author's work, either through their joint articles \cite{BoO0, BoO1, BoO2, BoO3} or through many informal conversations. 

We now state the main result of this article. 

\begin{thm}
\label{thm:Intro}
Let $A_1$, $A_2$, \dots, $A_k$ be $2$-by-$2$ matrices with coefficients in an algebra $\A$ over $\C$, such that:
\begin{enumerate}
\item each $A_i$ is triangular of the form 
$\bigl(
\begin{smallmatrix}
a_i&b_i \\ 0&a_i^{-1}
\end{smallmatrix}
\bigr)$
or
$\bigl(
\begin{smallmatrix}
a_i&0\\ b_i&a_i^{-1}
\end{smallmatrix}
\bigr)$
with $b_ia_i = q a_i b_i$ for some nonzero number $q\in \C-\{0\}$;
\item $a_i$ and $b_i$ commute with $a_j$ and $b_j$ whenever $i \neq j$. 

\end{enumerate}

Then, if $q^2$ is a primitive $n$--root of unity, 
$$
T_n( \Tr A_1A_2\dots A_k) = \Tr A_1^{(n)}A_2^{(n)}\dots A_k^{(n)}
$$
where each $A_i^{(n)}$ is obtained from $A_i$ by replacing $a_i$ and $b_i$ with their powers $a_i^n$ and $b_i^n$
\end{thm}

This statement is easier to understand if we illustrate it by an example. Consider the product of five triangular matrices
$$
A= 
\Big(
\begin{smallmatrix}
 a_1&b_1\\0& a_1^{-1}
\end{smallmatrix}
\Big)
\Big(
\begin{smallmatrix}
 a_2&b_2\\0& a_2^{-1}
\end{smallmatrix}
\Big)
\Big(
\begin{smallmatrix}
 a_3 &0\\b_3& a_3^{-1}
\end{smallmatrix}
\Big)
\Big(
\begin{smallmatrix}
 a_4&b_4\\0& a_4^{-1}
\end{smallmatrix}
\Big)
\Big(
\begin{smallmatrix}
 a_5&0\\b_5& a_5^{-1}
\end{smallmatrix}
\Big)
$$
where $b_ia_i = q a_i b_i$, and $a_i$ and $b_i$ commute with $a_j$ and $b_j$ whenever $i \neq j$.

Computing the product and taking the trace straightforwardly gives
\begin{align*}
\Tr A &= a_1a_2a_3a_4a_5 
+a_1 a_2a_3b_4b_5
+a_1 b_2b_3a_4a_5
+a_1 b_2b_3b_4b_5\\
&\qquad +a_1 b_2a_3^{-1}a_4^{-1}b_5
+b_1 a_2^{-1}b_3a_4a_5
+b_1 a_2^{-1}b_3b_4b_5\\
&\qquad +b_1 a_2^{-1}a_3^{-1}a_4^{-1}b_5
+a_1^{-1} a_2^{-1}b_3b_4a_5^{-1}
+a_1^{-1} a_2^{-1}a_3^{-1}a_4^{-1}a_5^{-1}.
\end{align*}

Since $\Tr A$ has 10 terms and the Chebyshev polynomial $T_n(t)$ has degree $n$, one would expect $T_n(\Tr A)$ to have about $10^n$ terms. However, when $q^2$ is a primitive $n$--root of unity, many cancellations occur and only 10 terms remain. In fact
\begin{align*}
T_n(\Tr A) &= a_1^n a_2^n a_3^n a_4^n a_5^n  
+a_1^n  a_2^n a_3^n b_4^n b_5^n 
+a_1^n  b_2^n b_3^n a_4^n a_5^n 
+a_1^n  b_2^n b_3^n b_4^n b_5^n \\
&\qquad +a_1^n  b_2^n a_3^{-n}a_4^{-n}b_5^n 
+b_1^n  a_2^{-n}b_3^n a_4^n a_5^n 
+b_1^n  a_2^{-n}b_3^n b_4^n b_5^n \\
&\qquad +b_1^n  a_2^{-n}a_3^{-n}a_4^{-n}b_5^n 
+a_1^{-n} a_2^{-n}b_3^n b_4^n a_5^{-n}
+a_1^{-n} a_2^{-n}a_3^{-n}a_4^{-n}a_5^{-n}\\
&= \Tr A^{(n)}
\end{align*}
where
$$
A^{(n)}= 
\Big(
\begin{smallmatrix}
 a_1^n&b_1^n\\0& a_1^{-n}
\end{smallmatrix}
\Big)
\Big(
\begin{smallmatrix}
 a_2^n&b_2^n\\0& a_2^{-n}
\end{smallmatrix}
\Big)
\Big(
\begin{smallmatrix}
 a_3^n &0\\b_3^n& a_3^{-n}
\end{smallmatrix}
\Big)
\Big(
\begin{smallmatrix}
 a_4^n&b_4^n\\0& a_4^{-n}
\end{smallmatrix}
\Big)
\Big(
\begin{smallmatrix}
 a_5^n&0\\b_5^n& a_5^{-n}
\end{smallmatrix}
\Big)
$$
is obtained from $A$ by replacing each $a_i$, $b_i$ with $a_i^n$, $b_i^n$.

When $q$ is transcendental,  there are only very few  cancellations and Proposition~\ref{prop:TnExactNumberMonomials}  shows that $T_n(\Tr A)$ is a sum of exactly $(5+\sqrt{24})^n + (5-\sqrt{24})^n >9.89^n$ monomials.  This explicit count is based on a positivity result (Proposition~\ref{prop:TnPositiveSum}) which may be of independent interest. 

Similarly, the example of Equation~(\ref{eqn:CancellationFirstExample}) is provided by applying Theorem~\ref{thm:Intro} to the matrix $A= 
\big(
\begin{smallmatrix}
 a_1&b_1\\0& a_1^{-1}
\end{smallmatrix}
\big)
\big(
\begin{smallmatrix}
 a_2&0\\b_2& a_2^{-1}
\end{smallmatrix}
\big)
$, setting $X=a_1a_2$ and $Y=b_1b_2$, and replacing $q$ with $\sqrt q$. 

The proof of Theorem~\ref{thm:Intro} essentially has two parts. The first step is representation theoretic, and connects $T_n(\Tr A)$ to the action of the matrix $A$ on the space $\A[X,Y]^q$ of polynomials in $q$--commuting variables $X$ and $Y$ and with coefficients in the algebra $\A$. This is a relatively easy adaptation to our context of a deep but classical result in the representation theory of the quantum group $\Uq$, the Clebsch-Gordan Decomposition. The second step is a simple computation of traces for this action of $A$ on  $\A[X,Y]^q$, which is much simpler than the original arguments of \cite{BonWon2}.

Among the hypotheses of Theorem~\ref{thm:Intro}, some are more natural than others. The $q$--com\-mutativity relations $b_ia_i = q a_i b_i$ are essentially mandated by the connection of the objects considered to the quantum group $\Uq$ and its dual Hopf algebra $\SLq$. Similarly, the requirement that the matrices $A_i$ be triangular is deeply tied to the structure of the Lie group $\SL(\C)$ and the quantum group $\Uq$ (and their Borel subgroups/subalgebras). The commutativity hypothesis that $a_i$ and $b_i$ commute with $a_j$ and $b_j$ whenever $i \neq j$ is less critical, and was introduced here to define $ \Tr A_1A_2\dots A_k \in \A$ (and the product matrix $A_1A_2\dots A_k\in \SLq(\A)$) in a straightforward way. In fact, it is possible to define such a trace without these commutativity properties, but this requires using the cobraiding of the Hopf algebra $\SLq$ as well as additional data that is reminiscent of the original topological context. This was implicitly done in \cite{BonWon1} for the Kauffman bracket skein algebra of a surface, but a quick comparison between the formulas of \cite[Lem.~21]{BonWon1} and \cite[Cor.~VIII.7.2]{Kas} should make it clear that these arguments can  be expanded to a more representation theoretic framework. The cancellations of Theorem~\ref{thm:Intro} then extend to this generalized setup, as in \cite{BonWon2}.

\section{The equation $(X+Y)^n=X^n + Y^n$}
\label{sect:FrobeniusEqn}

In spite of the first sentence of this article, most of our students do know the \emph{Binomial Formula}, which says that
\begin{align*}
(X+Y)^n &= X^n + \binom{n}{1}X^{n-1}Y+ + \binom{n}{2}X^{n-2}Y^2 + \dots + \binom{n}{n-1}XY^{n-1} +Y^n\\
&= \sum_{k=0}^n \binom{n}{k} X^{n-k}Y^k
\end{align*}
where $\binom{n}{k}$ is the \emph{binomial coefficient}
$$
\binom{n}{k} = \frac{n(n-1)(n-2) \dots (n-k+2)(n-k+1)}
{k(k-1)(k-2) \dots 2\,1}.
$$

If we are working in a ring $R$ with characteristic $n$ and if $n$ is prime (which in particular occurs when $R$ is a field), then $n=0$ in $R$ while $k(k-1)(k-2) \dots 2\,1\neq 0$ for every $k<n$. It follows that $\binom {n}{k}=0$ whenever $0<k<n$, so that the Frobenius relation
\newcounter{holdequation}
\setcounter{holdequation}{\theequation}
\begin{equation}
\setcounter{equation}{1}
(X+Y)^n = X^n + Y^n
\end{equation}
\setcounter{equation}{\theholdequation}
holds in this case.

Note that the hypothesis that the characteristic $n$ is prime is really necessary. For instance, in the ring $\Z/4$ of characteristic 4, 
$$
(X+Y)^4 = X^4 + 6X^2Y^2 + Y^4 \neq X^4 + Y^4. 
$$

A less well-known occurrence of Equation~(\ref{eqn:Frobenius})  involves variables $X$ and $Y$ that $q$--commute, in the sense that $YX = qXY$ for some $q\in \C$. The \emph{Quantum Binomial Formula} (see for instance \cite[\S IV.2]{Kas}) then states that
\begin{equation}
\label{eqn:QuantumBinomialFormula}
\begin{split}
(X+Y)^n &= X^n + \QB{n}{1}{q}X^{n-1}Y+ + \QB{n}{2}{q}X^{n-2}Y^2 + \dots + \QB{n}{n-1}{q} XY^{n-1} +Y^n\\
&= \sum_{k=0}^n \QB{n}{k}{q} X^{n-k}Y^k
\end{split}
\end{equation}
where $\QB{n}{k}{q}$ is the \emph{quantum binomial coefficient}
$$
\QB{n}{k}{q} = \frac{\QI{n}{q}\QI{n-1}{q}\QI{n-2}{q} \dots \QI{n-k+2}{q}\QI{n-k+1}{q}}
{\QI{k}{q}\QI{k-1}{q}\QI{k-2}{q} \dots \QI{2}{q}\QI{1}{q}}
$$
defined by the \emph{quantum integers}
$$
\QI{j}q = \frac{q^j-1}{q-1} = 1 +q+q^2 + \dots +q^{j-1}.
$$

We state the following property as a lemma, as we will frequently need to refer to it. 
\begin{lem}
\label{lem:QuantumBinomialZero}
If $q$ is a primitive $n$--root of unity, in the sense that $q^n=1$ and $q^k\neq 1$ whenever $0<k<n$, the quantum  binomial coefficient $\QB{n}{k}{q}$ is equal to $0$ for every $k$ with $0<k<n$. 
\end{lem}
\begin{proof}
Since $q$ is a primitive $n$--root of unity, $\QI{n}{q}=\frac{q^n-1}{q-1}=0$ while $\QI{k}{q}=\frac{q^k-1}{q-1}\neq0$ whenever $0<k<n$. The result follows. 
\end{proof}

\begin{cor}
If $YX=qXY$ with $q\in \C$ a primitive $n$--root of unity, the Frobenius relation
\setcounter{holdequation}{\theequation}
\begin{equation}
\setcounter{equation}{1}
(X+Y)^n = X^n + Y^n
\end{equation}
\setcounter{equation}{\theholdequation}
holds. \qed
\end{cor}

As in the characteristic $n$ case, it is necessary that $q$ be a \emph{primitive} $n$--root of unity. For instance, when  $YX=-XY$, $q=-1$ is a (non-primitive) 4--root of unity and $(X+Y)^4 = X^4 + 2X^2Y^2 + Y^4 \neq X^4 + Y^4$. 

We will now discuss generalizations of Equation~(\ref{eqn:Frobenius}) arising from properties of the quantum group $\Uq$. As indicated in the introduction, we will express these properties in terms of 2-by-2 matrices with coefficients in an algebra $\A$. 

Incidentally, all algebras considered in this article will be over $\C$. Other fields could be used, but the need for primitive $n$--roots of unity  make this convention more natural. 

We begin with the classical case of the algebra $\A=\C$, and of the Lie group $\SL(\C)$.

\section{The classical action of $\SL(\C)$ on $\C[X,Y]$}
\label{sect:SL2}

The \emph{special linear group} of order 2 is the group
$$
\SL(\C) = \big\{
\big(
\begin{smallmatrix}
 a&b\\c&d
\end{smallmatrix}
\big); a, b,c,d \in \C, ad-bc=1
\big\}
$$
of $2$-by-$ 2$ matrices with determinant 1. It has a left action on the plane $\C^2$, and therefore a right action by precomposition on the algebra 
\begin{align*}
\C[X,Y] &= \{ \text{polynomial functions on } \C^2\}\\
&= \Big\{\text{polynomials }
\sum_{i=0}^m \sum_{j=0}^n a_{ij} X^i Y^j; a_{ij} \in \C \Big\}
\end{align*}
of polynomials in the variables $X$ and $Y$. More precisely, the action of $\big(
\begin{smallmatrix}
 a&b\\c&d
\end{smallmatrix}
\big)\in \SL(\C)$ on $\C[X,Y]$ is such that 
\begin{equation}\label{eqn:ActionPlane}
P(X,Y)\big(
\begin{smallmatrix}
 a&b\\c&d
\end{smallmatrix}
\big) = P(aX+bY, cX + dY)
\end{equation}
for every polynomial $P(X,Y) \in \C[X,Y]$. 

This defines a map
$$
\rho \colon \SL(\C) \to \End_\C \big( \C[X,Y] \big)
$$
from $ \SL(\C) $ to the algebra of $\C$--linear maps $\C[X,Y] \to \C[X,Y]$. 

We collect a few elementary properties in the following lemma. 

\begin{lem}
$ $\pushQED{\qed}
\begin{enumerate}
\item For every $\big(
\begin{smallmatrix}
 a&b\\c&d
\end{smallmatrix}
\big) \in \SL(\C)$, $\rho\big(
\begin{smallmatrix}
 a&b\\c&d
\end{smallmatrix}
\big) \in \End_\C \bigl( \C[X,Y] \bigr)$ is also an algebra endomorphism of $\C[X,Y]$.

\item The map  $\rho$ is valued in the group $\Aut_\C \big( \C[X,Y] \big)$ of linear automorphisms of $\C[X,Y]$, and induces a group antihomomorphism
$
\rho \colon \SL(\C) \to \Aut_\C \big( \C[X,Y] \big)
$, in the sense that
$$
\rho\left(
\begin{pmatrix}
 a&b\\c&d
\end{pmatrix}
\begin{pmatrix}
 a'&b'\\c'&d'
\end{pmatrix}
\right)
=
\rho
\begin{pmatrix}
 a'&b'\\c'&d'
\end{pmatrix}
\circ
\rho
\begin{pmatrix}
 a&b\\c&d
\end{pmatrix}
$$
for every 
$\big(
\begin{smallmatrix}
 a&b\\c&d
\end{smallmatrix}
\big)$, 
$\big(
\begin{smallmatrix}
 a'&b'\\c'&d'
\end{smallmatrix}
\big)\in \SL(\C)$.

\item If $\C[X,Y]_n = \bigl\{ \sum_{i+j=n} a_{ij} X^iY^j; a_{ij} \in \C \bigr\}\cong \C^{n+1}$ denotes the vector space of homogeneous polynomials of degree $n$, the representation $\rho$ restricts to a finite-dimensional representation 
\begin{equation*}
\rho_n \colon \SL(\C) \to \Aut_\C \big( \C[X,Y] _n\big) \cong \Aut_\C(\C^{n+1}). \qedhere
\end{equation*}
\end{enumerate}
\end{lem}

The order reversal in the second conclusion reflects the fact that $\SL(\C)$ acts on $\C[X,Y]$ on the right. We could easily turn $\rho$ into a group homomorphism by composing it with any of the standard antiautomorphisms of $\SL(\C)$, such as 
$
\big(
\begin{smallmatrix}
 a&b\\c&d
\end{smallmatrix}
\big)
\mapsto
\big(
\begin{smallmatrix}
 a&b\\c&d
\end{smallmatrix}
\big)^{\mathrm t}
=
\big(
\begin{smallmatrix}
 a&c\\b&d
\end{smallmatrix}
\big)
$
or
$
\big(
\begin{smallmatrix}
 a&b\\c&d
\end{smallmatrix}
\big)
\mapsto
\big(
\begin{smallmatrix}
 a&b\\c&d
\end{smallmatrix}
\big)^{-1}
=
\big(
\begin{smallmatrix}
 d&-b\\-c&a
\end{smallmatrix}
\big)
$, and many authors do  this. For this reason, we will refer to $\rho$ and its restrictions $\rho_n$ as representations of $\SL(\C)$.

A classical property is that, up to isomorphism, the $\rho_n$ form the collection of all irreducible representations of $\SL(\C)$.

\section{The quantum plane and $\SLq(\A)$}
\label{sect:SL2q}

 For a nonzero number $q \in \C-\{0\}$, the  \emph{quantum plane} is the algebra $\C[X,Y]^q$ defined by two generators $X$ and $Y$, and by the relation $YX=qXY$. Namely, the elements of $\C[X,Y]^q$ are polynomials $P(X,Y)=\sum_{i=0}^m \sum_{j=0}^n a_{ij} X^i X^j$ that are multiplied using the relation $YX=qXY$. 

If we want to keep the property that  the matrices $\left(
\begin{smallmatrix}
 a&b\\c&d
\end{smallmatrix}\right) $ and $\left(
\begin{smallmatrix}
 a&c\\b&d
\end{smallmatrix}\right)$ act as algebra homomorphisms on $\C[X,Y]^q$, we need that
\begin{align*}
(cX+dY)(aX + bY) &= q (aX + bY)(cX+dY)\\
\text{and }
(bX+dY)(aX + cY) &= q (aX +cY)(bX+dY)
\end{align*}
to preserve the relation $YX=qXY$. Identifying the coefficients of $X^2$, $XY$ and $Y^2$, this would require
\begin{equation}\label{eqn:RelationsGL2q}
\begin{aligned}
 ba &= qab 
 &\qquad
 db&= qbd 
 &
 bc&=cb \\
 ca &= qac
 &
 dc&=qcd
&\qquad
ad- q^{-1} bc &= da - qbc
\end{aligned}
\end{equation}
which is clearly impossible if $q\neq 1$ and $a$, $b$, $c$, $d$ commute with each other. This leads us to the following definition (see for instance \cite[Chap.~IV]{Kas} for more background).

Given an algebra $\mathcal A$ over $\C$, let $\SLq(\A)$ denote the set of matrices
$
\left(
\begin{smallmatrix}
 a&b\\c&d
\end{smallmatrix}
\right)
$
where $a$, $b$, $c$, $d\in \mathcal A$ satisfy the relations of (\ref{eqn:RelationsGL2q}), as well as 
\begin{equation}
\label{eqn:QuantumDet1}
ad- q^{-1}bc =1.
\end{equation}

More formally, let $\SLq$ be the algebra defined by generators $a$, $b$, $c$, $d$ and by the relations of (\ref{eqn:RelationsGL2q}) and (\ref{eqn:QuantumDet1}). Then $\SLq(\A)$ can be interpreted as the set of all algebra homomorphisms $\SL^q \to \A$. 
The elements of $\SLq(\A)$ are called the \emph{$\mathcal A$--points of $\SLq$}.

Unlike $\SL(\C)$, the set $\SLq(\A)$ is far from being a group. It only comes with a partially defined multiplication. Indeed, if 
$
\big(
\begin{smallmatrix}
 a&b\\c&d
\end{smallmatrix}
\big)
$ and $
\big(
\begin{smallmatrix}
 a'&b'\\c'&d'
\end{smallmatrix}
\big) \in \SLq(\A)
$, the usual formula
\begin{equation}
\label{eqn:MatrixMultiply}
\bigg(
\begin{matrix}
 a&b\\c&d
\end{matrix}
\bigg)
\bigg(
\begin{matrix}
 a'&b'\\c'&d'
\end{matrix}
\bigg)
=
\bigg(
\begin{matrix}
 aa'+bc'&ab'+bd'\\ca'+ dc'&cb' + dd'
\end{matrix}
\bigg)
\end{equation}
gives an element of $\SLq(\A)$ only under additional hypotheses on the entries of these matrices, for instance if $a$, $b$, $c$, $d$ commute with $a'$, $b'$, $c'$, $d'$. This partially defined multiplication has an identity element 
$\left(
\begin{smallmatrix}
 1&0\\0&1
\end{smallmatrix}
\right) \in \SLq(\A)
$. However, the  operation of passing to the inverse somewhat misbehaves in the sense that the formal inverse $\big(
\begin{smallmatrix}
 a&b\\c&d
\end{smallmatrix}
\big)^{-1} = \big(
\begin{smallmatrix}
 d&-qb\\-q^{-1}c&a
\end{smallmatrix}
\big)
$ of $\big(
\begin{smallmatrix}
 a&b\\c&d
\end{smallmatrix}
\big) \in \SLq(\A)
$ is an element of $\SL^{q^{-1}}(\A)=\SLq(\A^\op)$ instead of $\SLq(\A)$; here $\A^\op$ is the \emph{opposite algebra} of the algebra $\A$, consisting of the vector space $\A$ endowed with the new multiplication $*_\op$ defined by the property that $a*_\op b=ba$ for every $a$, $b\in \A$. 
In general, the formula (\ref{eqn:MatrixMultiply}) gives a globally defined  multiplication $\SL(\A) \otimes \SL(\mathcal B) \to \SL(\A \otimes \mathcal B)$ for any two algebras $\A$ and~$\mathcal B$. 

In order to generalize to $\SLq(\A)$  the action of $\SL(\C)$ on $\C[X,Y]$, we introduce the  \emph{quantum $\A$--plane} as the algebra $\A[X,Y]^q= \A \otimes \C[X,Y]^q$. In practice the elements of $\A[X,Y]^q$ are polynomials $P(X,Y)=\sum_{i=0}^m \sum_{j=0}^n a_{ij} X^i X^j$ with coefficients $a_{ij} \in \A$, and are algebraically manipulated using the relation $YX=qXY$ while the variables $X$ and $Y$ commute with all elements of $\A$.

The relations defining $\SLq(\A)$ are specially designed that an $\mathcal A$--point of $\SLq$ acts as an algebra homomorphism on the quantum $\A$--plane $\A[X,Y]^q$, by the same formula (\ref{eqn:ActionPlane}) as in the commutative plane. More precisely, if $\End_\A (\A[X,Y]^q)$ denotes the space of $\A$--linear maps $\A[X,Y]^q \to \A[X,Y]^q$, we can define a map 
$$
\rho\colon \SLq(\A) \to \End_\A (\A[X,Y]^q)
$$
such that 
$$
\rho\left(
\begin{smallmatrix}
 a&b\\c&d
\end{smallmatrix}
\right) \Big(\sum_{i,j}a_{ij} X^i Y^j \Big) = \sum_{i,j}a_{ij} (aX+bY)^i (cX + dY)^j
$$
for every polynomial $\sum_{i,j}a_{ij}X^iY^j \in \A_q[X,Y]$. 

The map $\rho$ satisfies the following elementary properties.

\begin{lem}\label{lem:QuantumRep}$ $\pushQED{\qed}
\begin{enumerate}
\item For every $\left(\begin{smallmatrix}
 a&b\\c&d
\end{smallmatrix} \right)\in \SLq(\A)$, the $\A$--linear map
$
\rho\left(\begin{smallmatrix}
 a&b\\c&d
\end{smallmatrix} \right) \in \End_\A (\A[X,Y]^q)
$
is an algebra homomorphism.

\item For each $n$, the map $
\rho\left(\begin{smallmatrix}
 a&b\\c&d
\end{smallmatrix} \right) \in \End_\A (\A[X,Y]^q)
$ respects the space $\A[X,Y]_n^q = \A \otimes \C[X,Y]_n^q$ of homogeneous polynomials of degree $n$ in $X$ and $Y$ with coefficients in $\A$. As a consequence,  $\rho$ induces by restriction  a map 
\begin{equation*}
\rho_n\colon \SLq(\A) \to \End_\A (\A[X,Y]^q_n). 
\end{equation*}

\item If $a$, $b$, $c$, $d$ commute with  $a'$, $b'$, $c'$, $d'$ (so that the product $\big(\begin{smallmatrix}
 a&b\\c&d
\end{smallmatrix} \big) \big(\begin{smallmatrix}
 a'&b'\\c'&d'
\end{smallmatrix} \big) = \big(\begin{smallmatrix}
 aa'+bc' & ab' + bd'  \\ ca' + dc'& cb'+dd'
\end{smallmatrix} \big) \in \SLq(\A)$ makes sense), 
\begin{equation*}
\rho\left(
\begin{pmatrix}
 a&b\\c&d
\end{pmatrix}
\begin{pmatrix}
 a'&b'\\c'&d'
\end{pmatrix}
\right)
=
\rho\left(
\begin{matrix}
 a'&b'\\c'&d'
\end{matrix}
\right)
\circ
\rho\left(
\begin{matrix}
 a&b\\c&d
\end{matrix}
\right). \qedhere
\end{equation*}

\end{enumerate}
\end{lem}

\section{Traces and Chebyshev polynomials}

\subsection{Traces}
\label{sect:Traces}
 Traces can misbehave in the noncommutative context. However, we are  interested in endomorphisms of $\A[X,Y]_n^q = \A \otimes \C[X,Y]_n^q$, which have a natural trace.

To emphasize the key property needed, let $V$ be a finite dimensional vector space over $\C$, let $\A$ be an algebra over $\C$, and consider an $\A$--linear map $f \in \End_\A(\A \otimes V)$. The \emph{trace} of $f$ is defined as
$$
\Tr f = \sum_{i=0}^n a_{ii} \in \A
$$
where $e_0$, $e_1$, \dots, $e_n$ is a basis for the $\C$-vector space $V$, and where the coefficients $a_{ij} \in \A$ are defined by the property that $f(e_j) = \sum_{i=0}^n a_{ij}e_i$ in $\A \otimes V$ for every $j=0$, $1$, \dots, $n$. 

The usual commutative proof immediately extends to this context to give:
\begin{lem}
\label{lem:TraceInvariant}
The trace $\Tr f  \in \A$ is independent of the choice of the basis  $e_0$, $e_1$, \dots, $e_n$  for the $\C$-vector space $V$. \qed
\end{lem}

Note that this property would be false if we only required $e_0$, $e_1$, \dots, $e_n$ to be a basis for the $\A$--module $\A \otimes V \cong \A^{n+1}$.

For practice, let us carry out a few elementary computations for the representations $\rho_n \colon \SLq(\A) \to \End_\A\big(\A[X,Y]_n^q\big)$ of Lemma~\ref{lem:QuantumRep}. 

For $n=1$, consider $ \big(\begin{smallmatrix}
a&b\\c&d
\end{smallmatrix} \big)\in \SLq(\A)$ and its image $\rho_1 \big(\begin{smallmatrix}
a&b\\c&d
\end{smallmatrix} \big)\in  \End_\A\big(\A[X,Y]_1^q\big)$. The polynomials $X$, $Y$ form a basis for $\C[X,Y]_1^q$. Since $\rho_1\big(\begin{smallmatrix}
a&b\\c&d
\end{smallmatrix} \big)(X) = aX + bY$ and  $\rho_1\big(\begin{smallmatrix}
a&b\\c&d
\end{smallmatrix} \big)(Y) = cX + dY$ we conclude that  $\Tr \rho_1\big(\begin{smallmatrix}
a&b\\c&d
\end{smallmatrix} \big) $ is equal to $a+d$, namely to what we have implicitly called $\Tr \big(\begin{smallmatrix}
a&b\\c&d
\end{smallmatrix} \big)$ in the introduction. 

For $n=2$,  an elementary computation gives
\begin{align*}
\rho_2 \big(\begin{smallmatrix}
a&b\\c&d
\end{smallmatrix} \big) (X^2)
&= (aX+bY)^2
&&= a^2 X^2 + (1+q^2) abXY + b^2 Y^2\\
\rho_2 \big(\begin{smallmatrix}
a&b\\c&d
\end{smallmatrix} \big) (XY)
&= (aX+bY)(cX+dY)
\kern -50pt
&&=ac X^2 + (ad+qbc) XY + bd Y^2\\
\rho_2 \big(\begin{smallmatrix}
a&b\\c&d
\end{smallmatrix} \big) (Y^2)
&= (cX+dY)^2 
&&=c^2 X^2 + (1+q^2) cdXY + d^2 Y^2
\end{align*}
so that
\begin{align*}
\Tr \rho_2 \big(\begin{smallmatrix}
a&b\\c&d
\end{smallmatrix} \big) 
&= a^2 + (ad+qbc) + d^2\\
&= a^2 + ad + da +d^2 -1\\
&= (a+d)^2-1
\end{align*}
by remembering that $da-qbc=1$ from the definition of $\SLq(\A)$. 

For $n=3$, a longer but similar first step gives that
\begin{align*}
\Tr \rho_3 \big(\begin{smallmatrix}
a&b\\c&d
\end{smallmatrix} \big) 
&= a^3 + \big( a^2d+q(1+q^2) abc \big) + \big( a(1+q^2) bcd + ad^2 \big) + d^3\\
&= a^3 + a^2d + ada + q^2 ada -a -q^2a + dad + q^2 dad - d -q^2d + ad^2 +d^3,
\end{align*}
using again the property that $da-qbc=1$.

Since $qbc=da-1$ and $bca=q^2 abc$, we have that $da^2-a=q^2ada-q^2a$. Similarly, because $dbc=q^2 bcd$, $d^2a-a = q^2dad-q^2d$. Substituting these values in our first expression for $\Tr \rho_3 \big(\begin{smallmatrix}
a&b\\c&d
\end{smallmatrix} \big) $ gives
\begin{align*}
\Tr \rho_3 \big(\begin{smallmatrix}
a&b\\c&d
\end{smallmatrix} \big) 
&= a^3 + a^2d + ada +  da^2 + dad +  d^2a + ad^2 +d^3-2a-2d\\
&= (a+d)^3-2(a+d). 
\end{align*}

In all three cases, we have been able to express the trace $\Tr \rho_n \big(\begin{smallmatrix}
a&b\\c&d
\end{smallmatrix} \big) $ as a polynomial in $\Tr \big(\begin{smallmatrix}
a&b\\c&d
\end{smallmatrix} \big) =a+d$. We will see in Corollary~\ref{cor:Traces&ChebyshevSecondKind} that this is a general phenomenon.  Part of the purpose in the above calculations was to let the reader experience the fact that this property is not that easy to check by bare-hand computations, which justifies the introduction of the more theoretical constructions of \S \ref{sect:SL2q&Uq} and \S \ref{sect:ClebschGordan} below. 

\subsection{Chebyshev polynomials} 
\label{sect:Chebyshev}

We will encounter two types of Chebyshev polynomials. The \emph{(normalized) Chebyshev polynomials of the first kind} are the polynomials $T_n(t) \in\mathbb Z[t]$ recursively defined by 
\begin{equation}
\label{eqn:Chebyshev1Def}
\begin{split}
T_{n+1}(t) &= t\,T_n(t) - T_{n-1}(t)\\
T_1(t) &=t  \\
T_0(t)&=2.
\end{split}
\end{equation}

The \emph{(normalized) Chebyshev polynomials of the second kind} $S_n(t) \in\mathbb Z[t]$ are remarkably similar, and defined by 
\begin{equation}
\label{eqn:Chebyshev2Def}
\begin{split}
S_{n+1} &= t\,S_n(t) - S_{n-1}(t)\\
S_1(t) &=t  \\
S_0(t)&=1.
\end{split}
\end{equation}

In particular, in addition to $T_0(t)=2$, $S_0(t) =1$ and $T_1(t) = S_1(t)=t$, 
\begin{align*}
T_2(t) &= t^2-2 & S_2(t) &= t^2-1\\
T_3(t) &= t^3 -3t & S_3(t) &= t^3-2t\\
T_4(t) &= t^4-4t^2+2 & S_4(t) &= t^4-3t^2 +1\\
T_5(t) &= t^5-5t^3 +5t & S_5(t) &= t^5-4t^3 +3t\\
T_6(t) &= t^6-6t^4 +9t^2 -2 & S_6(t) &= t^6-5t^4 + 6t^2 -1
\end{align*}

\begin{lem}
\label{lem:ChebyshevFirstSecondKind}
The two types of Chebyshev polynomials are related by the property that
$$
T_n(t) = S_n (t) - S_{n-2}(t)
$$
for every $n\geq 2$
\end{lem}
\begin{proof}
This is an immediate consequence of the fact that the $T_n(t)$ and $S_n(t)$ satisfy the same linear recurrence relation, and of the initial conditions. 
\end{proof}

The following classical properties connect the Chebyshev polynomials $T_n(t)$ and $S_n(t)$ to the group $\SL(\C)$. 
\begin{lem}
\label{lem:ChebyshevTraces}
For every $A\in \SL(\C)$,
\begin{align*}
T_n (\Tr A) &= \Tr A^n \\
S_n (\Tr A)&= \Tr \rho_n(A) 
\end{align*}
where  $\rho_n \colon \SL(\C) \to \End \bigl( \C[X,Y]_n \bigr)$ is the $(n+1)$--dimensional representation of~{\upshape\S \ref{sect:SL2}}. 
\end{lem}

\begin{proof}
From the Cayley-Hamilton Theorem (or inspection)
$$
A^2 - (\Tr A )A+\mathrm{Id} =0
$$
for every $A=\big(
\begin{smallmatrix}
a&b\\c&d
\end{smallmatrix}
\big) \in \SL(\C)$. Multiplying both sides by $A^{n-1}$ and taking the trace, we see that $\Tr A^n$ satisfies the same recurrence relation
$$
\Tr A^{n+1} = (\Tr A) (\Tr A^n) - \Tr A^{n-1}
$$
as $T_n(\Tr A)$, as well as the same initial values for $n=0$ and $n=1$. It follows that $\Tr A^n = T_n(\Tr A)$ for every $n$.

For the property that $S_n (\Tr A)= \Tr \rho_n(A) $, which is not needed in this article, we can just refer to the special case $q=1$ of Corollary~\ref{cor:Traces&ChebyshevSecondKind} below. 
\end{proof}

Note the following closed form expression for the Chebyshev polynomial $T_n(t)\in \Z[t]$. 

\begin{lem}
\label{lem:ChebyshevClosedForm}
$$\textstyle
T_n(t) =\left( \frac{t+\sqrt{t^2-4}}2 \right)^n + \left( \frac{t-\sqrt{t^2-4}}2 \right)^n
$$
\end{lem}
\begin{proof}
Let $A \in \SL(\C)$ be a matrix such that $\Tr A=t$. 
If $t^2-4 \neq 0$, the characteristic polynomial $\lambda^2 -t\lambda+1$ of $A$ has two distinct roots $\frac{t\pm\sqrt{t^2-4}}2 $, which consequently are the eigenvalues of $A$. Then $A^n$ has eigenvalues $\left(\frac{t\pm\sqrt{t^2-4}}2\right)^n $ and, by Lemma~\ref{lem:ChebyshevTraces},
$$\textstyle
T_n(t) =\Tr A^n = \left( \frac{t+\sqrt{t^2-4}}2 \right)^n + \left( \frac{t-\sqrt{t^2-4}}2 \right)^n.
$$

By continuity, the property also holds when  $t^2-4 = 0$. 
\end{proof}

\subsection{The Hopf algebras $\SLq$ and $\Uq$}
\label{sect:SL2q&Uq}

 For most of the article, we are trying to keep the exposition at an elementary level in order to make the algebra more intuitive and to emphasize its connection with geometry. However, we now need deeper algebraic concepts and constructions, which will enable us to apply the well-known Clebsch-Gordan Decomposition for $\Uq$ (Theorem~\ref{thm:ClebschGordanUq}) to obtain a similar statement (Proposition~\ref{prop:ClebschGordanSL2q(A)}) for the set $\SLq(\A)$ of  2-by-2 matrices that we are interested in. 
 We follow here the conventions of \cite[Chap.~VI--VII] {Kas}. 

We already encountered the $\C$--algebra $\SLq$, defined by generators $a$, $b$, $c$, $d$ and by the relations of (\ref{eqn:RelationsGL2q}--\ref{eqn:QuantumDet1}). In particular, $\SLq(\A)$ is the set of  homomorphisms from $\SLq$ to the algebra~$\A$. 

A better known object is the quantum group $\Uq$, which is a deformation of the enveloping algebra of the Lie algebra $\mathfrak{sl}_2(\C)$ of $\SL(\C)$; see \cite{KiR, Dr1, Dr2, Jim}. Recall that $\Uq$ is defined by generators $E$, $F$, $K$, $K^{-1}$ and by the relations
\begin{equation}
 \begin{aligned}
 KK^{-1}&=K^{-1}K=1\qquad      &KE&=q^2EK\\
 EF-FE&=\frac{K-K^{-1}}{q-q^{-1}}      &KF&=q^{-2}FK
\end{aligned}
\end{equation}
This is a Hopf algebra, whose comultiplication $\Delta\colon \Uq \to \Uq \otimes \Uq$, counit $\epsilon \colon \Uq \to \C$ and antipode map $S \colon \Uq \to \Uq$ are respectively determined by the properties that
\begin{align*}
 \Delta(E) &= E\otimes K + 1\otimes E &
  \epsilon(E) &=0 &
  S(E)&= -EK^{-1}
 \\
 \Delta(F) &= F\otimes 1 + K^{-1} \otimes F &
   \epsilon(F) &=0 &
   S(F)&= -KF
 \\
 \Delta(K) &= K \otimes K&
\epsilon(K) &=1 &
 S(K)&= K^{-1}.
\end{align*}

Similarly,  $\SLq$ is a Hopf algebra with comultiplication $\Delta\colon \SLq \to \SLq \otimes \SLq$, counit $\epsilon \colon \SLq \to \C$ and antipode  $S \colon \SLq \to \SLq$ given by
\begin{align*}
 \Delta(a) &= a \otimes a + b\otimes c &
  \epsilon(a)&= 1 &
  S(a)&= d 
 \\
 \Delta(b) &= a\otimes b + b\otimes d &
   \epsilon(b)&= 0 &
   S(b)&= -qb
\\
 \Delta(c) &= c \otimes a + d \otimes c &
   \epsilon(c)&= 0 &
    S(c)&= -q^{-1}c 
\\
 \Delta(d) &= c \otimes b + d \otimes d &
   \epsilon(d)&= 1 &
    S(d)&=a.
\end{align*}

When $q=1$, the algebra $\SL^1$ is just the algebra of regular (= polynomial) functions $\SL(\C) \to \C$ on the algebraic group $\SL(\C)$. The comultiplication $\Delta\colon \SL^1 \to \SL^1 \otimes \SL^1$ then is the algebra homomorphism induced by the group multiplication $\SL(\C) \times \SL(\C) \to \SL(\C)$, the counit $\epsilon \colon \SL^1 \to \C$ is induced by the map $\{*\} \to \SL(\C)$ sending $*$ to the identity $\big( 
\begin{smallmatrix}
 1&0\\0&1
\end{smallmatrix}\big)$,
and the  antipode  $S \colon \SL^1 \to \SL^1$ is induced by $\big( 
\begin{smallmatrix}
a&b\\c&d
\end{smallmatrix}\big)
\mapsto 
\big( 
\begin{smallmatrix}
a&b\\c&d
\end{smallmatrix}\big)^{-1}$. 

Similarly, as $q\to 1$, the quantum group $\Uq$ converges to the enveloping algebra $\mathrm U(\mathfrak{sl}_2)$ of the Lie algebra of $\SL(\C)$, by consideration of $H=\frac{K-K^{-1}}{q-q^{-1}}$.

The relationship between $\SLq$ and $\Uq$ comes in the form of a linear pairing
\begin{equation}
\label{eqn:DualityDef}
 \begin{aligned}
 \langle\  ,\, \rangle \colon
 \Uq &\otimes \SLq &&\longrightarrow &&\C\\
 U &\otimes \alpha && \longmapsto &\langle U&, \alpha \rangle
\end{aligned}
\end{equation}
determined by the properties that
\begin{equation}
\label{eqn:DualityGenerators}
 \begin{aligned}
\langle E,a \rangle&=0 \quad&
\langle E,b \rangle&=1 \quad&
\langle E,c \rangle&=0 \quad&
\langle E,d \rangle&=0 \quad\\
\langle F,a \rangle&=0 \quad&
\langle F,b \rangle&=0 \quad&
\langle F,c \rangle&=1 \quad&
\langle F,d \rangle&=0 \quad\\
\langle K,a \rangle&=q \quad&
\langle K,b \rangle&=0 \quad&
\langle K,c \rangle&=0 \quad&
\langle K,d \rangle&=q^{-1}
\end{aligned}
\end{equation}
and 
\begin{align}
\label{eqn:DualityCompatibleMultComult1}
\langle U, \alpha \beta \rangle &= \sum_{(U)} \langle U',\alpha \rangle \langle U'' ,\beta \rangle\\
\label{eqn:DualityCompatibleMultComult2}
\langle UV, \alpha \rangle &= \sum_{(\alpha)} \langle U, \alpha' \rangle \langle V, \alpha'' \rangle
\end{align}
for every $\alpha$, $\beta \in \SLq$ and $U$, $V\in \Uq$, using Sweedler's notation that
\begin{align*}
 \Delta(U) &= \sum_{(U)} U' \otimes U'' \in \Uq \otimes \Uq\\
\text{and }  \Delta(\alpha) &= \sum_{(\alpha)} \alpha' \otimes \alpha'' \in \SLq \otimes \SLq
\end{align*}
for the comultiplications of $\SLq$ and $\Uq$. See Lemma~\ref{lem:Duality&FundamentalRepUq} for an interpretation of the formulas of (\ref{eqn:DualityGenerators}), and see \cite{Dr2, Tak1, Tak2}  and \cite[\S VII.4]{Kas} for  details. 

In particular, the duality $\langle \ , \, \rangle$ induces a linear map $\delta\colon\SLq \to \Uq^*$ from $\SLq$ to the dual of $\Uq$. We will need the following property. 

\begin{lem}[Takeuchi \cite{Tak2}]
 \label{lem:Takeuchi}
 The above duality map $\delta\colon\SLq \to \Uq^*$ is injective. \qed
\end{lem}

This is analogous to the property that, because the Lie group $\SL(\C)$ is connected, a regular function on $\SL(\C)$ is completely determined by its  derivatives and higher derivatives at the identity element. 

In general, the representation theory of a Lie algebra is significantly easier to analyze than that of the corresponding Lie group. The same phenomenon in the quantum world is one of the reasons why $\Uq$ is more popular than $\SLq$. 

In particular, there is an action 
$\sigma \colon \Uq \otimes \C[X,Y]^q \to \C[X,Y]^q $ of $\Uq$ over the quantum plane $\C[X,Y]^q$ by ``quantum derivation'', defined by the property that 
\begin{equation}
 \begin{split}
\sigma(E \otimes X^kY^l) &= \QQI{l}{q} X^{k+1} Y^{l-1} \qquad
\sigma(F \otimes X^kY^l) = \QQI{k}{q} X^{k-1} Y^{l+1} \\
\sigma(K \otimes X^kY^l) &= q^{k-l} X^{k} Y^{l} 
\end{split}
\end{equation}
where $\QQI kq$ denotes the other type of quantum integer 
$$
\QQI kq = \frac{q^k - q^{-k}}{q-q^{-1}} = q^{k-1} + q^{k-3} + \dots + q^{-k+3} + q ^{-k+1}= q^{-k+1} \QI{k}{q^2}.
$$

This action restricts to the space of homogeneous polynomials of degree $n$, and gives an $(n+1)$--dimensional representation
$$
\sigma_n \colon \Uq \to \End_\C\big( \C[X,Y]_n^q \big)
$$
for every $n$. 

When $q$ is not a root of unity, the $\sigma_n$ essentially realize all irreducible finite-dimensional representations of $\Uq$, up to isomorphism. To describe all irreducible finite-dimensional representations of $\Uq$, one just need one more  family of similar representations $\sigma_n'$, related to $\sigma_n$ by a simple sign twist. See for instance \cite[\S VI.2]{Kas}. 

Similarly, the representation $ \rho\colon \SLq(\A) \to \End_\A (\A[X,Y]^q) $ of \S \ref{sect:SL2q} comes from a coaction
$$
\tau \colon \C[X,Y]^q \to \SLq \otimes \C[X,Y]^q
$$ 
defined by the property  that 
$$
\tau \big(  P(X,Y) \big) = P(a\otimes X + b\otimes Y, c\otimes X + d \otimes Y)
$$
for every polynomial $P(X,Y) \in \C[X,Y]^q$. Indeed, if $A \in \SLq(\A)$ is considered as an algebra homomorphism $A \colon \SLq \to \A$, then $\rho(A) \in \End_\A (\A[X,Y]^q)$ is clearly the $\A$--linear extension of the $\C$--linear map $ \big(A \otimes \Id_{\C[X,Y]^q}\big) \circ \tau \colon \C[X,Y]^q \to \A[X,Y]^q$. 

The following statement relates the duality $\langle\ ,\, \rangle$ of (\ref{eqn:DualityDef}) to the 2--dimensional representation $\sigma_1\colon \Uq \to \End_\C\big( \C[X,Y]_1^q\big)$. 

\begin{lem}
\label{lem:Duality&FundamentalRepUq}
 For every $U\in \Uq$, the matrix of $\sigma_1(U) \in \End_\C\big( \C[X,Y]_1^q\big)$ in the basis $\{ X, Y\}$ is 
 $$
 \sigma_1(U) = 
\begin{pmatrix}
 \langle U,a\rangle &  \langle U,b\rangle\\
 \langle U,c\rangle &  \langle U,d\rangle
\end{pmatrix},
 $$
 where $a$, $b$, $c$, $d$ are the generators of $\SLq$. 
\end{lem}

\begin{proof}
 The property holds for $U=E$, $F$ or $K^{\pm1}$ by inspection in (\ref{eqn:DualityGenerators}), since $\sigma_1(E) = \big(
\begin{smallmatrix}
 0&1\\0&0
\end{smallmatrix}
 \big)$, $\sigma_1(F) = \big(
\begin{smallmatrix}
 0&0\\1&0
\end{smallmatrix}
 \big)$ and $\sigma_1(K) = \big(
\begin{smallmatrix}
 q&0\\0&q^{-1}
\end{smallmatrix}
 \big)$ in the basis $\{X,Y\}$. It then holds for any product of these generators by combining the fact that $\sigma_1$ is an algebra homomorphism, the compatibility of the duality $\langle\ ,\, \rangle$ with the multiplications and comultiplications given by (\ref{eqn:DualityCompatibleMultComult2}), and the definition of the comultiplication $\Delta\colon \SLq \to \SLq \otimes \SLq$.
\end{proof}

We now show that the duality $\langle\ ,\, \rangle$ connects the action of $\Uq$ on $\C[X,Y]^q$ to the coaction of $\SLq$. 

To see this, we first rewrite the action $\sigma \colon \Uq \otimes \C[X,Y]^q \to \C[X,Y]^q $ as a linear map $\Sigma \colon  \C[X,Y]^q \to \Uq^* \otimes \C[X,Y]^q $. 

\begin{lem}
\label{lem:ActionCoactionDualityCompatible}
 The action $\Sigma \colon  \C[X,Y]^q \to \Uq^* \otimes \C[X,Y]^q $, coaction $\tau \colon \C[X,Y]^q \to \SLq \otimes \C[X,Y]^q$ and duality map $\delta \colon \SLq \to \Uq^*$ are related by the property that the diagram
 $$
 \xymatrix{
  \C[X,Y]^q  
  \ar@/^20pt/[rrr]^\Sigma  
  \ar[r]_{\tau\qquad}
  & \SLq \otimes  \C[X,Y]^q \ar[rr]_{\delta\otimes \Id_{\C[X,Y]^q}}
  && \Uq^* \otimes\C[X,Y]^q
 }
 $$
is commutative. Namely,
$
\Sigma=(\delta\otimes \Id_{\C[X,Y]^q}) \circ  \tau
$.
\end{lem}

\begin{proof}
The property is equivalent to the commutativity of the diagram
$$
\xymatrixcolsep{60pt}
 \xymatrix{
  \Uq \otimes \C[X,Y]^q  
  \ar@/^20pt/[rr]^\sigma  
  \ar[r]_{\Id_{\Uq} \otimes \tau\qquad}
  & \Uq \otimes \SLq \otimes  \C[X,Y]^q
   \ar[r]_{\qquad\qquad\langle\  ,\ \rangle\otimes \Id_{\C[X,Y]^q}}
  &  \C[X,Y]^q
 }
 $$
 
 To simplify the formulas, write
$
 \phi = \Id_{\Uq} \otimes \tau
 $ and $
 \psi= \langle\  ,\, \rangle\otimes \Id_{\C[X,Y]^q}
$.
 We need to prove that 
\begin{equation}
\label{eqn:ActionCoactionDualityCompatible}
\psi \circ \phi \big (U\otimes P(X,Y) \big) = \sigma\big (U\otimes P(X,Y) \big)
\end{equation}
for every $U\in \Uq$ and $P(X,Y)\in \C[X,Y]^q$. 

We first consider the case where $P(X,Y)=X$. Then,
\begin{align*}
  \psi \circ \phi  (U\otimes X)
  &= \psi (U \otimes a \otimes X + U \otimes b \otimes Y)\\
  &=   \langle U , a \rangle  X + \langle U , b \rangle  Y\\
  &= \sigma_1(U)(X) = \sigma(U\otimes X)
\end{align*}
by Lemma~\ref{lem:Duality&FundamentalRepUq}. So, the property of (\ref{eqn:ActionCoactionDualityCompatible}) holds for $P(X,Y)=X$ and every $U\in \Uq$. 

An almost identical argument shows that (\ref{eqn:ActionCoactionDualityCompatible}) holds for $P(X,Y)=Y$ and every $U\in \Uq$. As a consequence,  (\ref{eqn:ActionCoactionDualityCompatible}) holds for every $P(X,Y) \in \C[X,Y]_1^q$ and  $U\in \Uq$.

We now claim that, if  (\ref{eqn:ActionCoactionDualityCompatible}) holds for $P(X,Y)$ and $Q(X,Y)$ and every $U\in \Uq$, then it holds for the product $P(X,Y)Q(X,Y)$ and every $U\in \Uq$. 

A fundamental property of the action $\sigma$ is that, in the terminology of \cite[\S V.6]{Kas},  it makes $\C[X,Y]^q$ a module algebra over $\Uq$. This means that, in addition to making $\C[X,Y]^q$ a module over the algebra $\Uq$,  $\sigma$ satisfies the ``quantum product rule'' that
\begin{equation}
\label{eqn:ModuleAlgebra}
 \sigma  \big (U\otimes P(X,Y)Q(X,Y) \big) 
= \sum_{(U)} \sigma  \big (U'\otimes P(X,Y) \big) \sigma  \big (U''\otimes Q(X,Y) \big) ,
\end{equation}
using Sweedler's notation that $\Delta(U) = \sum_{(U)} U' \otimes U''$. See  \cite[ \S VII.3]{Kas}.

Now, 
\begin{align*}
  \psi \circ \phi  \big (U\otimes P(X,Y)Q(X,Y) \big)
  &= \psi \big(U \otimes P(a \otimes X +  b \otimes Y ) Q(a \otimes X +  b \otimes Y ) \big)\\
  &=  \sum_{(U)} \psi \big(U' \otimes P(a \otimes X +  b \otimes Y ) \big) \psi \big(U'' \otimes  Q(a \otimes X +  b \otimes Y ) \big)\\
  &= \sum_{(U)} \sigma  \big (U'\otimes P(X,Y) \big) \sigma  \big (U''\otimes Q(X,Y) \big) \\
  &= \sigma  \big (U\otimes P(X,Y)Q(X,Y) \big) ,
 \end{align*}
where the second equality comes from (\ref{eqn:DualityCompatibleMultComult1}), the third equality reflects our hypothesis that $P(X,Y)$ and $Q(X,Y)$ satisfy (\ref{eqn:ActionCoactionDualityCompatible}) for every $U'''\in \Uq$, and the fourth equality results from (\ref{eqn:ModuleAlgebra}). This proves our claim that (\ref{eqn:ActionCoactionDualityCompatible}) holds for $P(X,Y)Q(X,Y)$ and for every $U\in \Uq$.

This inductive step proves that (\ref{eqn:ActionCoactionDualityCompatible}) holds in all cases, and concludes the proof of Lemma~\ref{lem:ActionCoactionDualityCompatible}. 
\end{proof}

\subsection{The Clebsch-Gordan Decomposition for $\SLq$ and $\SLq(\A)$}
\label{sect:ClebschGordan}

 A great feature of Hopf algebras is that their comultiplication $\Delta$ enables one to take the tensor product of two representations. 

For $\Uq$, the \emph{Quantum Clebsch-Gordan Decomposition} expresses the action of $\Uq$ on the tensor product $\C[X,Y]^q_m \otimes \C[X,Y]^q_n$ as a direct sum of (irreducible) representations over $\C[X,Y]^q_{m+n-2k}$ with $0\leq k \leq \inf\{m,n\}$. See \cite {KoK, KiR, Vak}, and \cite[\S VII.7]{Kas}. 

We state here the result for the case we need, when $m=1$.  Recall that the action of $\Uq$ on $\C[X,Y]^q$ restricts to an algebra homomorphism $ \sigma_n \colon \Uq \to \End_\C\big( \C[X,Y]_n^q \big)$ for every $n$.  Also, the tensor product 
 $$\sigma_1 \otimes \sigma_n \colon \Uq \to \End_\C \big(  \C[X,Y]^q_1 \otimes \C[X,Y]^q_n \big) = \End_\C \big(  \C[X,Y]^q_1\big)  \otimes   \End_\C \big(  \C[X,Y]^q_n \big) $$
  is defined by the property that 
  $$
  \sigma_1 \otimes \sigma_n(U) = \sum_{(U)} \sigma_1(U') \otimes \sigma_n(U'')
  $$
  for every $U\in \Uq$, using Sweedler's notation that $\Delta(U) = \sum_{(U)} U' \otimes U''$. 

\begin{thm}
 [Clebsch-Gordan Decomposition for $\Uq$]
 \label{thm:ClebschGordanUq}
 When $q$ is not a $k$--root of unity with $k\leq n$, there exists a $\C$--linear isomorphism $\phi \colon \C[X,Y]^q_1 \otimes \C[X,Y]^q_n \to \C[X,Y]^q_{n+1} \oplus \C[X,Y]^q_{n-1}$ such that the diagram
 $$
 \xymatrix{
 \C[X,Y]^q_1 \otimes \C[X,Y]^q_n 
 \ar[rrr]^{\sigma_1 \otimes \sigma_n(U)} 
 \ar[d]_\phi^\cong
 &&&   \C[X,Y]^q_1 \otimes \C[X,Y]^q_n 
 \ar[d]^\phi_\cong
 \\
 \C[X,Y]^q_{n+1} \oplus \C[X,Y]^q_{n-1}  \ar[rrr]_{\sigma_{n+1}(U) \oplus \sigma_{n-1}(U)} 
 &&& \C[X,Y]^q_{n+1} \oplus \C[X,Y]^q_{n-1}
 }
 $$
 commutes for every $U\in \Uq$. \qed
 \end{thm}

 We now consider the coaction $\tau \colon  \C[X,Y]^q \to \SLq \otimes  \C[X,Y]^q$, and more precisely its restriction $\tau_n \colon  \C[X,Y]^q_n \to \SLq \otimes  \C[X,Y]^q_n$. 
 
 Tensor products of coactions are much simpler to define, and
 $$
 \tau_1 \otimes \tau_n  \big(P\otimes Q \big) =  \tau_1 (P) \otimes \tau_n(Q) \in  \SLq \otimes \C[X,Y]^q_1 \otimes \C[X,Y]^q_n 
 $$
for every $P \in  \C[X,Y]^q_1 $ and $Q\in  \C[X,Y]^q_n $.

We now use the duality $\langle\ ,\,\rangle$ to deduce the following result from Theorem~\ref{thm:ClebschGordanUq}. 

\begin{prop}
  [Clebsch-Gordan Decomposition for $\SLq$]
  \label{prop:ClebschGordanSL2q}
  Suppose that  $q$ is not a $k$--root of unity with $k\leq n$, and consider 
 the $\C$--linear isomorphism $\phi \colon \C[X,Y]^q_1 \otimes \C[X,Y]^q_n \to \C[X,Y]^q_{n+1} \oplus \C[X,Y]^q_{n-1}$ of Theorem~{\upshape \ref{thm:ClebschGordanUq}}. Then,  the diagram
 $$
 \xymatrix{
 \C[X,Y]^q_1 \otimes \C[X,Y]^q_n 
 \ar[rr]^{\tau_1 \otimes \tau_n\quad} 
 \ar[d]_\phi^\cong
 &&   \SLq\otimes \C[X,Y]^q_1 \otimes \C[X,Y]^q_n 
 \ar[d]^{\Id_{\SLq}\otimes\phi}_\cong\\
 \C[X,Y]^q_{n+1} \oplus \C[X,Y]^q_{n-1}  
 \ar[rr]_{\tau_{n+1} \oplus \tau_{n-1}\quad} 
 && \SLq\otimes \C[X,Y]^q_{n+1} \oplus \C[X,Y]^q_{n-1}
 }
 $$
 commutes, in the sense that $(\Id_{\SLq}\otimes\phi) \circ (\tau_1 \otimes \tau_n) = (\tau_{n+1} \oplus \tau_{n-1}) \circ \phi $
\end{prop}

\begin{proof} To simplify the notation, set $V_k = \C[X,Y]^q_k$. Then the commutative diagram of Lemma~\ref{lem:ActionCoactionDualityCompatible} restricts for each $k$ to a commutative diagram 
\begin{equation}
\label{eqn:ClebschGordanSL2q1}
 \xymatrix{
  V_k  
  \ar@/^20pt/[rrrr]^{\Sigma_k}  
  \ar[rr]_{ \tau_k}
  &&  \SLq \otimes  V_k \ar[rr]_{\delta \otimes \Id_{V_k}}
 & &\Uq^*\otimes V_k
 }
\end{equation}
 where $\Sigma_k$ is related to the action $\sigma_k \colon \Uq \to \End_\C(V_k)$  by the property that, for each $U\in \Uq$ and $P \in V_k$, the element $\sigma_k(U)(P) \in V_k$ is obtained by evaluating $\Sigma_k(P) \in \Uq^* \otimes V_k$ at $U$. 

In the diagram 
\begin{equation}
\label{eqn:ClebschGordanSL2q2}
\xymatrix{
V_1 \otimes V_n 
\ar@/^20pt/[rrrr]^{\Sigma_1 \otimes \Sigma_n}
\ar[rr]_{\tau_1 \otimes \tau_n\quad}
\ar[d]_\phi
&& \SLq \otimes V_1 \otimes V_n 
\ar[rr]_{\delta \otimes \Id_{V_1 \otimes V_n}\quad}
\ar[d]^{\Id_{\SLq}\otimes \phi}
&& \Uq^* \otimes V_1 \otimes V_n 
\ar[d]^{\Id_{\Uq^*}\otimes \phi}
\\
V_{n+1} \oplus V_{n-1} 
\ar@/_20pt/[rrrr]_{\Sigma_{n+1} \oplus \Sigma_{n-1}}
\ar[rr]^{\tau_{n+1} \oplus \tau_{n-1}\quad}
&& \SLq \otimes (V_{n+1} \oplus V_{n-1} )
\ar[rr]^{\delta\otimes \Id_{V_{n+1} \oplus V_{n-1}}\quad}
&& \Uq^* \otimes (V_{n+1} \oplus V_{n-1} )
}
\end{equation}
we want to show that the left-hand square commutes. Here,  $\Sigma_1 \otimes \Sigma_n$ is defined  to be related to the tensor product $\sigma_1 \otimes \sigma_n \colon \Uq \to \End_\C(V_1 \otimes V_n)$ by the property that, for each $U\in \Uq$, $P_1 \in V_1$ and $P_n \in V_n$, the element $\sigma_1 \otimes \sigma_n(U)(P_1 \otimes P_n) \in V_1 \otimes V_n$ is obtained by evaluating $\Sigma_1 \otimes \Sigma_n(P_1 \otimes P_n) \in \Uq^* \otimes V_1 \otimes V_n$ at $U$. The map $\Sigma_{n+1} \oplus \Sigma_{n-1}$ is similarly associated to the direct sum  $\sigma_{n+1} \oplus \sigma_{n-1} \colon \Uq \to \End_\C(V_{n+1} \oplus V_{n-1})$.

Because of the way $\Sigma_1 \otimes \Sigma_n$ and $\Sigma_{n+1} \oplus \Sigma_{n-1}$ are respectively associated to $\sigma_1 \otimes  \sigma_n$ and   $\sigma_{n+1} \oplus \sigma_{n-1}$, Theorem~\ref{thm:ClebschGordanUq} shows that the outer rectangle of  (\ref{eqn:ClebschGordanSL2q2}) commutes, in the sense that
$$
(\Id_{\Uq^*}\otimes \phi) \circ (\Sigma_1 \otimes \Sigma_n) = (\Sigma_{n+1} \oplus \Sigma_{n-1}) \circ \phi.
$$

The lower triangle 
\begin{equation*}
\xymatrix{
V_{n+1} \oplus V_{n-1} 
\ar@/_20pt/[rrrr]_{\Sigma_{n+1} \oplus \Sigma_{n-1}}
\ar[rr]^{\tau_{n+1} \oplus \tau_{n-1}\qquad}
&& \SLq \otimes (V_{n+1} \oplus V_{n-1} )
\ar[rr]^{\delta\otimes \Id_{V_{n+1} \oplus V_{n-1}}\quad}
&& \Uq^* \otimes (V_{n+1} \oplus V_{n-1} )
}
\end{equation*}
of (\ref{eqn:ClebschGordanSL2q2}) commutes by an immediate application of (\ref{eqn:ClebschGordanSL2q1}). 

The commutativity of the upper triangle requires more thought, because  tensor products of actions of algebras are more complicated than direct sums.

\begin{lem}
\label{lem:TensorProdCompatibleDuality}
 The diagram
 \begin{equation*}
\xymatrix{
V_1 \otimes V_n 
\ar@/^20pt/[rrrr]^{\Sigma_1 \otimes \Sigma_n}
\ar[rr]_{\tau_1 \otimes \tau_n\quad}
&& \SLq \otimes V_1 \otimes V_n 
\ar[rr]_{\delta \otimes \Id_{V_1 \otimes V_n}\quad}
&& \Uq^* \otimes V_1 \otimes V_n 
}
\end{equation*}
 commutes.
\end{lem}

\begin{proof} Because of the relationships between $\Sigma_1 \otimes \Sigma_n$ and the action $\sigma_1 \otimes \sigma_n$, and between the map $\delta \colon \SLq \to \Uq^*$ and the duality $\langle\ , \, \rangle \colon \Uq \otimes \SLq \to \C$, it suffices to show that
 \begin{equation}
 \label{eqn:ClebschGordanSL2q3}
\xymatrix{
V_1 \otimes V_n 
\ar@/^20pt/[rrrr]^{\sigma_1 \otimes \sigma_n(U)}
\ar[rr]_{\tau_1 \otimes \tau_n\quad}
&& \SLq \otimes V_1 \otimes V_n 
\ar[rr]_{\qquad\langle U, \ \rangle \quad}
&&   V_1 \otimes V_n 
}
\end{equation}
commutes for every $U\in \Uq$. Here, for a vector space $V$, we shorten the notation and  write $\langle U, \ \rangle \colon \SLq \otimes V \to V$ for the map that we previously denoted by $\langle U, \ \rangle \otimes \Id_V$. 

This property is an immediate consequence of the fact (\ref{eqn:DualityCompatibleMultComult1}--\ref{eqn:DualityCompatibleMultComult2}) that $\langle\ , \, \rangle$ establishes a duality between multiplications and comultiplications. 
Indeed, given two polynomials $P_1 \in V_1$ and $P_n \in V_n$,
\begin{align*}
\sigma_1 \otimes \sigma_n(U)(P_1 \otimes P_n) &= \sum_{(U)} \sigma_1(U')(P_1) \otimes \sigma_n(U'') (P_n)\\
 &=  \sum_{(U)} \langle U', \tau_1(P_1)\rangle \otimes \langle U'', \tau_n(P_n) \rangle \\
 &= \langle U, \tau_1(P_1) \otimes \tau_n(P_n) \rangle 
\end{align*}
 where the first equality reflects the definition of $\sigma_1 \otimes \sigma_2$, the second equality comes from Lemma~\ref{lem:ActionCoactionDualityCompatible} (or (\ref{eqn:ClebschGordanSL2q1})), and the third equality follows from (\ref{eqn:DualityCompatibleMultComult1}). 
 
 The proves the commutativity of (\ref{eqn:ClebschGordanSL2q3}), and therefore Lemma~\ref{lem:TensorProdCompatibleDuality}. 
\end{proof}
 
 We are now ready to conclude the proof of Proposition~\ref{prop:ClebschGordanSL2q}. We proved that, in the diagram (\ref{eqn:ClebschGordanSL2q2}), the outer rectangle and the two upper and lower triangles commute. By Lemma~\ref{lem:Takeuchi}, the map $\delta \colon \SLq \to \Uq^*$ is injective. It easily follows that the left-hand square
 \begin{equation*}
\xymatrix{
V_1 \otimes V_n 
\ar[rr]_{\tau_1 \otimes \tau_n\quad}
\ar[d]^\phi
&& \SLq \otimes V_1 \otimes V_n 
\ar[d]^{\Id_{\SLq}\otimes \phi}
\\
V_{n+1} \oplus V_{n-1} 
\ar[rr]^{\tau_{n+1} \oplus \tau_{n-1}\quad}
&& \SLq \otimes (V_{n+1} \oplus V_{n-1} )
}
\end{equation*}
 commutes. This is exactly what we needed to prove. 
\end{proof}

After this long digression through the Hopf algebras $\SLq$ and $\Uq$, we now return to our original topic of interest, namely the set $\SLq(\A)$ of $\A$--points of $\SLq$ for some algebra $\A$. 

The representation $ \rho_n\colon \SLq(\A) \to \End_\A (\A[X,Y]^q_n) $ is related to the coaction
$\tau_n \colon \C[X,Y]_n^q \to \SLq \otimes \C[X,Y]_n^q$ 
by the property that, if an $\A$--point $A \in \SLq(\A)$ is considered as an algebra homomorphism $A \colon \SLq \to \A$, then $\rho_n(A) \in \End_\A (\A[X,Y]_n^q)$ is  the $\A$--linear extension of the $\C$--linear map $ \big(A \otimes \Id_{\C[X,Y]_n^q}\big) \circ \tau \colon \C[X,Y]_n^q \to \A[X,Y]_n^q$. 

\begin{prop}
\label{prop:ClebschGordanSL2q(A)}
 When $q$ is not a $k$--root of unity with $k\leq n$, the representation $$\rho_1 \otimes_\A \rho_n \colon \SLq(\A) \to \End_\A \big( \A[X,Y]_1^q \otimes_\A \A[X,Y]_n^q \big)$$ is isomorphic  over $\C$ to the direct sum $\rho_{n+1} \oplus \rho_{n-1}$ of the representations $\rho_{n+1} \colon \SLq(\A) \to \End \big(  \A[X,Y]_{n+1}^q \big)$ and $\rho_{n-1} \colon \SLq(\A) \to \End \big(  \A[X,Y]_{n-1}^q \big)$. Namely, there exists a $\C$--linear isomorphism
$$
\phi \colon   \C[X,Y]_1^q \otimes_\C \C[X,Y]_n^q  \to \C[X,Y]_{n+1}^q  \oplus \C[X,Y]_{n-1}^q  ,
$$
inducing an $\A$--linear isomorphism
$$
\Id_\A \otimes_\C \phi \colon   \A[X,Y]_2^q \otimes_\A \A[X,Y]_n^q \to  \A[X,Y]_{n+1}^q  \oplus \A[X,Y]_{n-1}^q   ,
$$
such that the diagram
$$
\xymatrix{
 \A[X,Y]_1^q \otimes_\A \A[X,Y]_n^q 
 \ar[d]_{\Id_\A \otimes_\C \phi}^\cong
 \ar[rrr]^{\rho_1(A) \otimes_\A \rho_n(A)}
& &&  \A[X,Y]_1^q \otimes_\A \A[X,Y]_n^q 
 \ar[d]^{\Id_\A \otimes_\C \phi}_\cong
 \\
 \A[X,Y]_{n+1}^q  \oplus \A[X,Y]_{n-1}^q
 \ar[rrr]_{\rho_{n+1}(A) \oplus \rho_{n-1}(A)}
 &&&  \A[X,Y]_{n+1}^q  \oplus \A[X,Y]_{n-1}^q
}
$$
commutes for every $A\in \SLq(\A)$. 
\end{prop}

\begin{proof}
 This is an immediate consequence of Proposition~\ref{prop:ClebschGordanSL2q} and of the relationship between the coactions $\tau_k \colon \C[X,Y]_k^q \to \SLq \otimes_\C \C[X,Y]_k^q $ and the $\A$--linear maps  $\rho_k(A) \in \End_\A (\A[X,Y]_k^q)$. 
\end{proof}

The fact that the isomorphism $  \A[X,Y]_1^q \otimes_\A \A[X,Y]_n^q  \to \A[X,Y]_{n+1}^q  \oplus \A[X,Y]_{n-1}^q $ between $\rho_1 \otimes_\A \rho_n$ and $\rho_{n+1} \oplus \rho_{n-1}$  comes from a $\C$--linear isomorphism  $  \C[X,Y]_1^q \otimes_\C \C[X,Y]_n^q  \to \C[X,Y]_{n+1}^q  \oplus \C[X,Y]_{n-1}^q $  will be crucial for our consideration of traces in the next section. 

\subsection{The trace of $\rho_n(A)$ for $A\in \SLq(\A)$} After the hard work of \S \ref{sect:SL2q&Uq} and \S \ref{sect:ClebschGordan}, we now have appropriate tools to compute for $A\in \SLq(\A)$ the trace of $\rho_n(A)$ in terms of the trace of $A$. The few computations that we did at the end of \S \ref{sect:Traces} should convince the reader that this result would be hard to obtain without the heavy machinery of \S\S  \ref{sect:SL2q&Uq}--\ref{sect:ClebschGordan}. 

\begin{cor}
\label{cor:Traces&ChebyshevSecondKind}
For every $
\big(
\begin{smallmatrix}
 a&b\\c&d
\end{smallmatrix}
\big) 
\in \SLq(\A)$,
$$
\Tr \rho_n
\big(
\begin{smallmatrix}
 a&b\\c&d
\end{smallmatrix}
\big) 
=
S_n (a+d)
$$
where $S_n$ is the $n$--th Chebyshev polynomial of the second kind. 
\end{cor}
\begin{proof} The property makes sense for all $q$, but we first restrict attention to the case where $q$ is not a root of unity in order to apply Proposition~\ref{prop:ClebschGordanSL2q(A)}. 

Consider $A= \big(
\begin{smallmatrix}
 a&b\\c&d
\end{smallmatrix}
\big) \in \SLq(\A)$. By Proposition~\ref{prop:ClebschGordanSL2q(A)}, the $\A$--linear maps  $\rho_1(A) \otimes_\A \rho_n(A)$ and $\rho_{n+1}(A) \oplus \rho_{n-1}(A)$ are isomorphic over $\C$. By Lemma~\ref{lem:TraceInvariant}, they consequently have the same trace. 
Therefore,
\begin{align*}
\big( \Tr \rho_1(A)\big)\big( \Tr \rho_n(A)\big) &= \Tr \big(\rho_1(A) \otimes_\A \rho_n(A)\big)\\
 &= \Tr \big(\rho_{n+1}(A) \oplus \rho_{n-1}(A) \big)\\
 &= \Tr \big(\rho_{n+1}(A)\big) + \Tr\big( \rho_{n-1}(A)\big)
\end{align*}
and $\Tr\rho_n(A)$ therefore satisfies the same recurrence relation (\ref{eqn:Chebyshev1Def}--\ref{eqn:Chebyshev2Def}) as the Chebyshev polynomials. 

For $n=0$, $\rho_0(A)=\Id_\A$ and $\Tr \rho_0(A)=1$. By definition of the Chebyshev polynomial $S_n(t)$ in (\ref{eqn:Chebyshev2Def}), it follows that 
$$
\Tr \rho_n(A) = S_n\big( \Tr \rho_1(A) \big). 
$$

We already computed $\Tr \rho_1(A)$ in \S \ref{sect:Traces}. If $A= \big(
\begin{smallmatrix}
 a&b\\c&d
\end{smallmatrix}
\big)$, then $\rho_1(A)$ has matrix $ \big(
\begin{smallmatrix}
 a&c\\b&d
\end{smallmatrix}
\big)$ in the $\C$--basis $\{X,Y\}$ for $\A[X,Y]_1^q = \A \otimes \C[X,Y]_1^q$. (Note that this is the transpose matrix.) It follows that $\Tr \rho_1(A) = a+d$, which concludes our computation when $q$ is not a root of unity. 

The case where $q$ is a root of unity follows from this generic case by continuity. To justify this continuity argument for an arbitrary algebra $\A$, first consider the case when $\A=\SLq$, where we can make sense of continuity with respect to $q$ (for instance by considering $\SLq$ as an algebra over $\C[q, q^{-1}]$). The algebra $\SLq$ admits a tautological $\SLq$--point $I \in \SLq(\SLq)$, defined by the identity algebra homomorphism $I \colon \SLq \to \SLq$. Then, $ \Tr \rho_n(I) = S_n( \Tr I ) \in \SLq$ for every $q$ by continuity from the case where $q$ is not a root of unity. 

For a general algebra $\A$ and an $\A$--point $A \in \SLq(\A)$, a little thought will convince the reader that  $\Tr \rho_n(A)\in \A$ is the image of $ \Tr \rho_n(I)  \in \SLq$ under the algebra homomorphism $A \colon \SLq \to \A$; in particular, $\Tr A = A( \Tr I)$ by specialization to the case $n=1$. Then, 
\begin{align*}
 \Tr \rho_n(A) &= A \big( \Tr \rho_n(I) \big) = A \big( S_n( \Tr I ) \big)\\
 &= S_n \big( A( \Tr I ) \big) = S_n (\Tr A)
\end{align*}
for every $q$, using the fact that $A$ is an algebra homomorphism for the third equality. 
\end{proof}

\section{Miraculous cancellations}
\label{sect:Cancellations}

We now prove the main result of this article, namely Theorem~\ref{thm:Intro} which we rephrase a Theorem~\ref{thm:MiraculousCancel} below. Although we just encountered Chebyshev polynomials $S_n(t)$ of the second kind, the property involves the Chebyshev polynomials $T_n(t)$ of the first kind. 

Note that, if an $\A$--point $A=\bigl(
\begin{smallmatrix}
a&b \\ c&d
\end{smallmatrix}
\bigr) \in \SLq(\A)$ of $\SLq$ is upper triangular, namely is such that $c=0$, then necessarily $d=a^{-1}$ by the quantum determinant relation $ad-q^{-1}bc=1$ of Relation (\ref{eqn:QuantumDet1}). As a consequence, $A$ can be written as $A=\bigl(
\begin{smallmatrix}
a&b \\ 0 & a^{-1}
\end{smallmatrix}
\bigr)$ with $ba=qab$. Similarly, any lower triangular element of $\SLq(\A)$ is of the form $\bigl(
\begin{smallmatrix}
a&0 \\ b & a^{-1}
\end{smallmatrix}
\bigr)$ with $ba=qab$. 

\begin{thm}
\label{thm:MiraculousCancel}
Let $A_1$, $A_2$, \dots, $A_k\in \SLq(\A)$ be $\mathcal A$--points of $\SLq$ such that:
\begin{enumerate}
\item each $A_i$ is triangular of the form 
$\bigl(
\begin{smallmatrix}
a_i&b_i \\ 0&a_i^{-1}
\end{smallmatrix}
\bigr)$
or
$\bigl(
\begin{smallmatrix}
a_i&0\\ b_i&a_i^{-1}
\end{smallmatrix}
\bigr)$ for some $a_i$, $b_i\in \A$ 
(with $b_ia_i = q a_i b_i$);
\item $a_i$ and $b_i$ commute with $a_j$ and $b_j$ whenever $i \neq j$, so that we can make sense of  the product $A_1A_2 \dots A_n \in \SLq(\A)$. 
\end{enumerate}
Then, if $q^2$ is a primitive $n$--root of unity,
$$
T_n \big( \Tr A_1A_2\dots A_{k-1}A_k \big ) = \Tr A_1^{(n)}A_2^{(n)} \dots A_{k-1}^{(n)} A_k^{(n)}
$$
where, for each $i$,  $A_i^{(n)} =
\bigl(
\begin{smallmatrix}
a_i^n&b_i^n \\ 0&a_i^{-n}
\end{smallmatrix}
\bigr)$
or
$\bigl(
\begin{smallmatrix}
a_i^n&0\\ b_i^n &a_i^{-n}
\end{smallmatrix}
\bigr)$ is the $\mathcal A$--point of  $\mathrm{SL}_2^{q^{n^2}}$ obtained from $A_i =\bigl(
\begin{smallmatrix}
a_i&b_i \\ 0&a_i^{-1}
\end{smallmatrix}
\bigr)$
or
$\bigl(
\begin{smallmatrix}
a_i&0\\ b_i&a_i^{-1}
\end{smallmatrix}
\bigr)$ by replacing  $a_i$ and $b_i$ with their powers $a_i^n$ and $b_i^n$, respectively.
\end{thm}

Note that $b_i^n a_i^n = q^{n^2} a_i^n b_i^n$ since $b_ia_i = qa_ib_i$. Also, $q^{n^2}$ is equal to $\pm1$ since $q^{2n}=1$, and  is always $+1$ when $n$ is even.

\begin{proof} For notational convenience, we will reverse the indexing and prove the equivalent statement that
\begin{equation}
 \label{eqn:MiraculousCancelReversed}
  T_n \big( \Tr A_kA_{k-1}\dots A_2A_1 \big ) = \Tr A_k^{(n)}A_{k-1}^{(n)} \dots A_2^{(n)} A_1^{(n)}.
\end{equation}
For this, we will use Lemma~\ref{lem:ChebyshevFirstSecondKind} and Corollary~\ref{cor:Traces&ChebyshevSecondKind}, so that 
\begin{align*}
 T_n \big( \Tr A_kA_{k-1}\dots A_2A_1 \big ) &= S_n \big( \Tr A_kA_{k-1}\dots A_2A_1 \big ) - S_{n-2} \big( \Tr A_kA_{k-1}\dots A_2A_1 \big )\\
 &=\Tr \rho_n ( A_kA_{k-1}\dots A_2A_1  ) - \Tr \rho_{n-2} (  A_kA_{k-1}\dots A_2A_1 )
\end{align*}
for the representations $ \rho_m\colon \SLq(\A) \to \End_\A \big(\A[X,Y]^q_m \big) $ of \S \ref{sect:SL2q}. 

We first compute these traces.

When $A_i$ is lower triangular,  the image of $X^{n-u}Y^u \in \A[X,Y]^q_n$ under $\rho_n(A_i)$ is 
\begin{align*}
\rho_n(A_i)(X^{n-u} Y^{u} )&= 
\rho_n\left(\begin{smallmatrix}
 a_i& 0\\ b_i& a_i^{-1}
\end{smallmatrix}\right)
(X^{n-u} Y^{u})
=(a_iX)^{n-u} (b_iX + a_i^{-1}Y)^{u}\\
&=
a_i^{n-u} X^{n-u}\sum_{v=0}^{u}  \QB{u}{v}{q^2}  b_i^{u-v} X^{u-v}a_i^{-v}Y^{v}\\
&=
\sum_{v=0}^{u}  \QB{u}{v}{q^2}  q^{-v(u-v)}
a_i^{n-u-v} b_i^{u-v} X^{n-v}Y^{v},
\end{align*}
using the Quantum Binomial Formula (\ref{eqn:QuantumBinomialFormula}) of \S \ref{sect:FrobeniusEqn}. 
In particular, if we express $\rho_n(A_i)$ in the basis $\{X^{n-u}Y^u; u=0, 1, \dots, n\}$ for $\A[X,Y]^q_n$, the entries of the corresponding matrix are
\begin{equation}
\label{eqn:UpperTriangMatrix}
\rho_n(A_i)_{vu} = 
\begin{cases}
 \QB{u}{v}{q^2}  q^{-v(u-v)}
a_i^{n-u-v} b_i^{u-v} &\text{ if } v\leq u\\
0 &\text{ if }v>u.
\end{cases}
\end{equation}
Note that this matrix is upper triangular.

Similarly, when $A_i$ is upper triangular, 
\begin{align*}
\rho_n(A_i)(X^{n-u} Y^{u} )&= 
\rho_n\left(\begin{smallmatrix}
 a_i&b_i\\0& a_i^{-1}
\end{smallmatrix}\right)
(X^{n-u} Y^{u})
=(a_iX+ b_iY)^{n-u} ( a_i^{-1}Y)^u\\
&=
\sum_{v=u}^{n}  \QB{n-u}{n-v}{q^2}  q^{-u(v-u)}
a_i^{n-u-v} b_i^{v-u} X^{n-v}Y^v
\end{align*}
and
\begin{equation}
\label{eqn:LowerTriangMatrix}
\rho_n(A_i)_{vu} = 
\begin{cases}
0 &\text{ if } v< u\\
 \QB{n-u}{n-v}{q^2}  q^{-u(v-u)}
a_i^{n-u-v} b_i^{v-u}  &\text{ if }v\geq u.
\end{cases}
\end{equation}

In particular,
\begin{align}
\label{eqn:TraceDimN}
\Tr \rho_n(A_kA_{k-1}\dots A_2A_1)&= \Tr \rho_n(A_1) \circ \rho_n(A_2) \circ \dots \circ \rho_n(A_{k-1}) \circ \rho_n(A_k) \\ \notag
&= \hskip -10pt
\sum_{u_1,u_2, \dots, u_k \in \{0, \dots, n\}} \hskip -10pt
\rho_n(A_1)_{u_1u_2}\, \rho_n(A_2)_{u_2u_3} \dots\\ \notag
&\qquad\qquad\qquad \qquad\qquad \dots \rho_n(A_{k-1})_{u_{k-1}u_k}\,\rho_n(A_k)_{u_ku_1}
\end{align}
and
\begin{align}
\label{eqn:TraceDimN-2}
\Tr \rho_{n-2}(A_kA_{k-1}\dots A_2A_1)&= \hskip -10pt
\sum_{v_1,v_2, \dots, v_k \in \{0, \dots, {n-2}\}} \hskip -10pt
\rho_{n-2}(A_1)_{v_1v_2}\, \rho_{n-2}(A_2)_{v_2v_3} \dots\\ \notag
&\qquad\qquad\qquad \qquad\qquad
 \dots \rho_{n-2}(A_{k-1})_{v_{k-1}v_k}\,\rho_{n-2}(A_k)_{v_kv_1},
\end{align}
where the terms $\rho_n(A_i)_{vu} $ are given by Equations (\ref{eqn:UpperTriangMatrix}) and (\ref{eqn:LowerTriangMatrix}). 

We distinguish three types of terms in the sum of Equation~(\ref{eqn:TraceDimN}), according to the corresponding indices $u_1$, $u_2$, \dots, $u_k \in \{0, \dots, n\}$:
\begin{enumerate}
\item [({i})] no $u_i$ is equal to $0$ or $n$;
\item [({ii})] some but not all $u_i$ are equal to $0$ or $n$;
\item [({iii})] all $u_i$ are equal to $0$ or $n$;
\end{enumerate}

We begin with the first type. 
\begin{lem}
\label{lem:NoIndex0orN}
If no $u_i$ is equal to $0$ or $n$, the term
$$
U_n(u_1, \dots, u_k)=\rho_n(A_1)_{u_1u_2}\, \rho_n(A_2)_{u_2u_3} \dots \rho_n(A_{k-1})_{u_{k-1}u_k}\,\rho_n(A_k)_{u_ku_1}
$$
of Equation~{\upshape(\ref{eqn:TraceDimN})} corresponding to  $u_1$, $u_2$, \dots, $u_k \in \{1, \dots, n-1\}$ is equal to the term
$$
U_{n-2}(v_1, \dots, v_k) =\rho_{n-2}(A_1)_{v_1v_2}\, \rho_{n-2}(A_2)_{v_2v_3} 
\dots \rho_n(A_{k-1})_{v_{k-1}v_k}\,\rho_{n-2}(A_k)_{v_kv_1}
$$
of Equation~{\upshape(\ref{eqn:TraceDimN-2})} corresponding to the indices $v_1$, $v_2$, \dots, $v_k \in \{0, \dots, n-2\}$ with $v_i=u_i-1$. 
\end{lem}

\begin{proof} Set $u_{k+1}=u_1$ and $v_{k+1}=v_1$ to introduce uniformity in the notation.

If $A_i$ is lower triangular, Equation~(\ref{eqn:UpperTriangMatrix}) gives
$$
\rho_n(A_i)_{u_iu_{i+1}} = 
\begin{cases}
 \QB{u_{i+1}}{u_i}{q^2}  q^{-u_i(u_{i+1}-u_i)}
a_i^{n-u_{i+1}-u_i} b_i^{u_{i+1}-u_i} &\text{ if } u_i\leq u_{i+1}\\
0 &\text{ if }u_i>u_{i+1}
\end{cases}
$$
while, using the property that $v_i=u_i-1$,
$$
\rho_{n-2}(A_i)_{v_iv_{i+1}} = 
\begin{cases}
 \QB{u_{i+1}-1}{u_i-1}{q^2}  q^{-(u_i-1)(u_{i+1}-u_i)}
a_i^{n-u_{i+1}-u_i} b_i^{u_{i+1}-u_i}  &\text{ if } u_i\leq u_{i+1}\\
0 &\text{ if }u_i>u_{i+1}.
\end{cases}
$$
Since  $\QB{u}{v}{q^2}= \frac{\QI{u}{q^2}}{\QI{v}{q^2}} \QB{u-1}{v-1}{q^2} $, it follows that
$$
\rho_n(A_i)_{u_iu_{i+1}} = 
\frac{\QI{u_{i+1}}{q^2}}{\QI{u_i}{q^2}}q^{u_i-u_{i+1}} \rho_{n-2}(A_i)_{v_iv_{i+1}}
$$
when $A_i$ is lower triangular. 

Similarly, when $A_i$ is upper triangular,
$$
\rho_n(A_i)_{u_iu_{i+1}} = 
\begin{cases}
0 &\text{ if } u_i< u_{i+1}\\
 \QB{n-u_{i+1}}{n-u_i}{q^2}  q^{-u_{i+1}(u_i-u_{i+1})}
a_i^{n-u_{i+1}-u_i} b_i^{u_i-u_{i+1}}  &\text{ if }u_i\geq u_{i+1}.
\end{cases}
$$
and
$$
\rho_{n-2}(A_i)_{v_iv_{i+1}} = 
\begin{cases}
0 &\text{ if } u_i< u_{i+1}\\
 \QB{n-u_{i+1}-1}{n-u_i-1}{q^2}  q^{-(u_{i+1}-1)(u_i-u_{i+1})}
a_i^{n-u_{i+1}-u_i} b_i^{u_i-u_{i+1}}  &\text{ if }u_i\geq u_{i+1},
\end{cases}
$$
Therefore
\begin{align*}
\rho_n(A_i)_{u_iu_{i+1}} &= 
\frac{\QI{n-u_{i+1}}{q^2}}{\QI{n-u_i}{q^2}}q^{u_{i+1}-u_i} \rho_{n-2}(A_i)_{v_iv_{i+1}}\\
&= \frac{-q^{-2u_{i+1}}\QI{u_{i+1}}{q^2}}{-q^{-2u_i}\QI{u_i}{q^2}}q^{u_{i+1}-u_i} \rho_{n-2}(A_i)_{v_iv_{i+1}}\\
&= \frac{\QI{u_{i+1}}{q^2}}{\QI{u_i}{q^2}}q^{u_i-u_{i+1}} \rho_{n-2}(A_i)_{v_iv_{i+1}}
\end{align*}
using the property that
$$
\QI{n-u}{q^2} = \frac{q^{2n-2u}-1}{q^2-1} =  -q^{-2u}\frac{q^{2u}-1}{q^2-1} = -q^{-2u} \QI{u}{q^2}
$$
since $q^{2n}=1$.

As a consequence, we get the same formula whether $A_i$ is upper or lower triangular. Taking the product over all $i$,
\begin{align*}
 U_n(u_1, \dots, u_k) &= U_{n-2}(v_1, \dots, v_k) \prod_{i=1}^k \frac{\QI{u_{i+1}}{q^2}}{\QI{u_i}{q^2}}q^{u_i-u_{i+1}}  =U_{n-2}(v_1, \dots, v_k) ,
\end{align*}
where the second equality comes from the fact that $u_{k+1}=u_1$. This proves Lemma~\ref{lem:NoIndex0orN}. 
\end{proof}

\begin{lem}
\label{lem:SomeIndices0orN}
If some but not all indices $u_i$ are equal to $0$ or $n$, the term
$$
U_n(u_1, \dots, u_k)=\rho_n(A_1)_{u_1u_2}\, \rho_n(A_2)_{u_2u_3} \dots \rho_n(A_{k-1})_{u_{k-1}u_k}\,\rho_n(A_k)_{u_ku_1}
$$
of Equation~{\upshape(\ref{eqn:TraceDimN})} corresponding to  $u_1$, $u_2$, \dots, $u_k \in \{0, \dots, n\}$ is equal to 0. 
\end{lem}
\begin{proof} This is a consequence of Lemma~\ref{lem:QuantumBinomialZero}, which says that, because $q^2$ is a primitive $n$--root of unity, the quantum binomial coefficient $\QB{n}{u}{q^2}$ is equal to 0 for $0<u<n$. 

For convenience, set $u_{k+1}=u_1$ as in the proof of Lemma~\ref{lem:NoIndex0orN}. By hypothesis, there is then an index $i$ such that $0<u_i<n$ and $u_{i+1}=0$ or $n$. 

Consider first the case when $0<u_i<n$ and $u_{i+1}=0$. If $A_i$ is lower triangular, then $\rho_n(A_i)_{u_iu_{i+1}}=0$ by Equation~(\ref{eqn:UpperTriangMatrix}), and consequently $U_n(u_1, \dots, u_k)=0$. 
Otherwise, Equation~(\ref{eqn:LowerTriangMatrix}) gives
$$
\rho_n(A_i)_{u_iu_{i+1}} =  \QB{n}{n-u_i}{q^2}  a_i^{n-u_i} b_i^{u_i} =0
$$
by Lemma~\ref{lem:QuantumBinomialZero}. This proves that $U_n(u_1, \dots, u_k)=0$ in this case. 

Similarly, if  $0<u_i<n$ and $u_{i+1}=n$, Equation~(\ref{eqn:LowerTriangMatrix}) immediately shows that $U_n(u_1, \dots, u_k)=0$ when $A_i$ is upper triangular, and otherwise gives 
$$
\rho_n(A_i)_{u_iu_{i+1}} =  \QB{n}{u_i}{q^2}  q^{-u_i(n-u_i)}
a_i^{-u_i} b_i^{n-u_i} =0
$$
by Lemma~\ref{lem:QuantumBinomialZero}, again proving that $U_n(u_1, \dots, u_k)=0$. 
\end{proof}

Lemmas~\ref{lem:NoIndex0orN} and \ref{lem:SomeIndices0orN} show that, when computing
$$
T_n(\Tr A_kA_{k-1}\dots A_2A_1) = \Tr \rho_n(A_kA_{k-1}\dots A_2A_1) - \Tr \rho_{n-2}(A_kA_{k-1}\dots A_2A_1)
$$
using the expressions of Equations(\ref{eqn:TraceDimN}--\ref{eqn:TraceDimN-2}), the only terms left are 
$$\sum_{u_1,u_2, \dots, u_k \in \{0, n\}} \hskip -10pt
\rho_n(A_1)_{u_1u_2}\, \rho_n(A_2)_{u_2u_3} \dots
 \rho_n(A_{k-1})_{u_{k-1}u_k}\,\rho_n(A_k)_{u_ku_1}$$
where
$$
\rho_n(A_i)_{u_iu_{i+1}} = 
\begin{cases}
a_i^{n} &\text{ if } u_i= u_{i+1}=0\\
b_i^{n} &\text{ if } u_i=0 \text{ and } u_{i+1}=n\\
0 &\text{ if }u_i =n \text{ and } u_{i+1}=0\\
a_i^{-n} &\text{ if } u_i= u_{i+1}=n
\end{cases}
$$
if $A_i$ is lower triangular, and 
$$
\rho_n(A_i)_{u_iu_{i+1}} = 
\begin{cases}
a_i^{n}  &\text{ if }u_i= u_{i+1}=0\\
0 &\text{ if } u_i=0 \text{ and } u_{i+1}=n\\
 b_i^{n}  &\text{ if }u_i =n \text{ and } u_{i+1}=0\\
a_i^{-n} &\text{ if }u_i= u_{i+1}=n
\end{cases}
$$
if $A_i$ is upper triangular. 

As a consequence, comparing the general case to the case $n=1$,
\begin{align*}
 T_n(\Tr A_kA_{k-1}\dots A_2A_1) &= \hskip -10pt\sum_{u_1,u_2, \dots, u_k \in \{0, 1\}} \hskip -10pt
\rho_1\big(A_1^{(n)}\big)_{u_1u_2}\, \rho_1\big(A_2^{(n)}\big)_{u_2u_3} \dots  \\
&\qquad\qquad\qquad\qquad\qquad
\dots \rho_1\big(A_{k-1}^{(n)}\big)_{u_{k-1}u_k}\,\rho_1\big(A_k^{(n)}\big)_{u_ku_1}\\
&= \Tr \rho_1\big(A_k^{(n)} A_{k-1}^{(n)} \dots A_2^{(n)} A_1^{(n)}\big)
\end{align*}
where $A_i^{(n)} = \Big( 
\begin{smallmatrix}
 a_i^n & b_i^n \\b_i^n & d_i^n
\end{smallmatrix}
 \Big)$ is obtained from $A_i = \big( 
\begin{smallmatrix}
 a_i & b_i \\b_i & d_i
\end{smallmatrix}
 \big)$
 by replacing $a_i$, $b_i$, $b_i$, $d_i$ with $a_i^n$, $b_i^n$, $b_i^n$, $d_i^n$, respectively. 
 
 We already observed that, for an $\A$--point $A = \big( 
\begin{smallmatrix}
 a & b \\c& d
\end{smallmatrix}
 \big) \in \SLq(\A)$, the matrix of $\rho_1(A)$ in the basis $\{X,Y\}$ for $\C[X,Y]_1^q$ is the transpose $\big( 
\begin{smallmatrix}
 a & c \\b& d
\end{smallmatrix}
 \big) $, so that $\Tr \rho_1(A) = a+d = \Tr A$. It follows that
 \begin{align*}
 T_n(\Tr A_kA_{k-1}\dots A_2A_1) &=  \Tr \rho_1\big(A_k^{(n)} A_{k-1}^{(n)} \dots A_2^{(n)} A_1^{(n)}\big)\\
 &=  \Tr A_k^{(n)} A_{k-1}^{(n)} \dots A_2^{(n)} A_1^{(n)}.
\end{align*}
This is exactly the relation (\ref{eqn:MiraculousCancelReversed}) that we wanted to prove, which concludes the proof of Theorem~\ref{thm:MiraculousCancel}. 
\end{proof}

\section{A positivity property}

Many fewer cancellations occur when $q$ is not a root of unity. This can be precisely quantified by using  a certain positivity property for $T_n \big( \Tr A_1A_2\dots A_k \big )$.

Let $A_1A_2 \dots A_k$ be a product of triangular matrices $A_i=  \big( 
\begin{smallmatrix}
 a_i & b_i\\ 0&a_i^{-1}
\end{smallmatrix}\big)$ or 
$\big(\begin{smallmatrix}
 a_i & 0\\ b_i&a_i^{-1}
\end{smallmatrix}
 \big) \in \SLq(\A)$ satisfying the hypotheses of  Theorems~\ref{thm:Intro} or \ref{thm:MiraculousCancel}, except that we are not requiring $q$ to be a root of unity. Then, $\Tr A_1A_2\dots A_k$ can be written as a sum of monomials $\prod_{i=1}^k c_i$ with $c_i\in \{a_i, b_i, a_i^{-1}\}$ and therefore, for every polynomial $P(t) \in \Z[t]$ with integer coefficients,   $P \big( \Tr A_1A_2\dots A_k \big )$ can be written as a sum of monomials of the form $\pm q^\xi \prod_{i=1}^k a_i^{\alpha_i} b_i^{\beta_i} $ with integer powers $\xi$, $\alpha_i$, $\beta_i \in \Z$ (with $\beta_i\geq 0$).
 
 The following result states that, when $P(t)$ is one of the Chebyshev polynomials $S_n(t)$ or $T_n(t)$,  the signs $\pm$ can always be taken to be $+$.

\begin{prop}
\label{prop:TnPositiveSum}
 Under the hypotheses of  Theorems~{\upshape\ref{thm:Intro}} or {\upshape\ref{thm:MiraculousCancel}} but without any assumption on the parameter $q\in \C-\{0\}$,  the evaluations  $S_n \big( \Tr A_1A_2\dots A_k \big )$ and $T_n \big( \Tr A_1A_2\dots A_k \big ) \in \A$ of the Chebyshev polynomials can be written as a sum of positive monomials of the form $+ q^\xi \prod_{i=1}^k a_i^{\alpha_i} b_i^{\beta_i} $ with integer powers $\xi$, $\alpha_i$, $\beta_i \in \Z$.
\end{prop}

\begin{proof} The case of $S_n(t)$ is  relatively simple. We computed
$$
S_n \big( \Tr A_1A_2\dots A_k \big ) = \Tr \rho_n ( A_1A_2\dots A_k )
$$
in the course of the proof of Theorem~\ref{thm:MiraculousCancel}. In particular, Equations (\ref{eqn:UpperTriangMatrix}--\ref{eqn:TraceDimN}) show that, with no assumption on $q$, $\Tr \rho_n ( A_1A_2\dots A_k )$ is a sum of positive monomials $+ q^\xi \prod_{i=1}^k a_i^{\alpha_i} b_i^{\beta_i} $. Indeed, it is well-known (and also follows from the Quantum Binomial Formula (\ref{eqn:QuantumBinomialFormula}))  that the quantum binomial coefficients $\QB uv{q^2}$ are polynomials in $q^2$ with nonnegative integer coefficients. 

The proof for $T_n(t)$ is more elaborate. We want to show that, when computing 
$$
T_n \big( \Tr A_1A_2\dots A_k \big ) = \Tr \rho_n ( A_1A_2\dots A_k ) - \Tr \rho_{n-2} ( A_1A_2\dots A_k ),
$$
each monomial of $ \Tr \rho_{n-2} ( A_1A_2\dots A_k )$ cancels out with a monomial of $\Tr \rho_n ( A_1A_2\dots A_k ) $; this can be seen as a weaker form of Lemma~\ref{lem:NoIndex0orN}. For this, we will give a different computation of $ \Tr \rho_{n} ( A_1A_2\dots A_k )$. 

This computation goes back to the principles underlying the Quantum Binomial Formula. Let $\C\langle X,Y \rangle$ be the free algebra generated by the set $\{X,Y\}$. Namely, $\C\langle X,Y \rangle$ consists of all formal polynomials $P(X,Y)$ in noncommutating variables $X$ and $Y$, and these polynomials are multiplied without simplifications. In particular, the quantum plane $\C[X,Y]^q$ is the quotient of $\C \langle X,Y \rangle$ by the ideal generated by $YX-qXY$, which gives a natural projection $\pi \colon \C \langle X,Y \rangle \to \C[X,Y]^q$. 

Similarly, consider $\A\langle X,Y \rangle = \A \otimes \C\langle X,Y \rangle$, and the projection $\Id_\A \otimes \pi$ which we will also denote as $\pi \colon  \A \langle X,Y \rangle \to \A[X,Y]^q$ for short.

Let $\A\langle X,Y \rangle_n$ be the linear subspace of $\A\langle X,Y \rangle $ consisting of all homogeneous polynomials of degree $n$. Namely,  $\A\langle X,Y \rangle _n$ consists of all finite sums
$$
P(X,Y) = \sum_u \alpha_u Z_{u1} Z_{u2} \dots Z_{un}
$$
where  the coefficients $\alpha_u $ are in $ \A$ and where each variable $Z_{uv}$ is equal to $X$ or to $Y$. In particular, $\A\langle X,Y \rangle_n$ is isomorphic to $\A^{2^n}$ as an $\A$--module. For comparison, remember that $\A[X,Y]^q_n$ is isomorphic to $\A^{n+1}$. 

The representation $\rho_n \colon \SLq(\A) \to \End_\A \big( \A[X,Y]^q_n \big)$ lifts to a representation $\widehat \rho_n \colon \SLq(\A) \to \End_\A \big( \A\langle X,Y\rangle_n \big)$ defined by the property that
$$
\widehat \rho_n \big( \begin{smallmatrix}
a&b\\c&d
\end{smallmatrix} \big) \big(P(X,Y) \big) = P(aX+bY, cX + dY )
$$
for every homogeneous polynomial $P(X,Y) \in \A \langle X,Y \rangle$  of degree $n$ in the noncommuting variables $X$ and $Y$ (and with coefficients in $\A$). 

To give a more combinatorial description of this action, note that every element of $\A\langle X,Y\rangle_n$ can be uniquely written as a sum of monomials  $\alpha Z_1Z_2 \dots Z_n$ where $\alpha \in \A$ and  each $Z_u \in \{X,Y\}$. Then, $\widehat \rho_n\big( 
\begin{smallmatrix}
 a&b\\c&d
\end{smallmatrix}
 \big) (\alpha Z_1Z_2 \dots Z_n) \in \A \langle X,Y \rangle_n$ is the sum of all monomials $\alpha' Z_1'Z_2' \dots Z_n'$ obtained from $\alpha Z_1Z_2 \dots Z_n$ by replacing each $Z_u$ with:
\begin{itemize}
 \item[\textbullet] either $aX$ or $bY$, if $Z_u=X$;
 \item[\textbullet]  either $cX$ or $dY$, if $Z_u=Y$,
\end{itemize}
(and pushing all coefficients of $\A$ to the front). 

As a consequence, 
$$
 \widehat \rho_n(A_1A_2 \dots A_k) (\alpha Z_1Z_2 \dots Z_n)
 =  \widehat \rho_n(A_k) \circ  \widehat \rho_n(A_{k-1}) \circ \dots \circ  \widehat \rho_n(A_1) (\alpha Z_1Z_2 \dots Z_n)
$$
can be described as follows. Let $\mathcal M(\alpha Z_1Z_2 \dots Z_n)$ be the set of all  sequences $M =(M_i)_{i=0, 1, \dots, k}$ of monomials $M_i=\alpha_i Z_{i1}Z_{i2} \dots Z_{in} \in \A \langle X,Y \rangle_n$ such that
\begin{enumerate}
 \item $M_0 = \alpha Z_1Z_2 \dots Z_n$;
 \item $M_i$ is obtained from $M_{i-1}$ by replacing each $Z_{(i-1)u}$ with
\begin{itemize}
 \item[\textbullet]   $a_i X$ or $b_iY$ if $Z_{(i-1)u}=X$ and $A_i = \left( 
\begin{smallmatrix}
 a_i & b_i \\ 0& a_i^{-1}
\end{smallmatrix}
 \right)$;
 
  \item[\textbullet]   $a_i^{-1} Y$ if $Z_{(i-1)u}=Y$ and $A_i = \left( 
\begin{smallmatrix}
 a_i & b_i \\ 0& a_i^{-1}
\end{smallmatrix}
 \right)$;
 
  \item[\textbullet]   $a_i X$  if $Z_{(i-1)u}=X$ and $A_i = \left( 
\begin{smallmatrix}
 a_i & 0 \\ b_i& a_i^{-1}
\end{smallmatrix}
 \right)$;
 
  \item[\textbullet]   $b_i X$  or $a_i^{-1} Y$ if $Z_{(i-1)u}=Y$ and $A_i = \left( 
\begin{smallmatrix}
 a_i & 0 \\ b_i& a_i^{-1}
\end{smallmatrix}
 \right)$.

\end{itemize}
\end{enumerate}
Then, as we expand $ \widehat \rho_n(A_1A_2 \dots A_k) (\alpha Z_1Z_2 \dots Z_n)\in \A\langle X,Y\rangle_n$, we see that its  decomposition  into monomials is given by
$$
 \widehat \rho_n(A_1A_2 \dots A_k) (\alpha Z_1Z_2 \dots Z_n) = \sum_{M \in \mathcal M(\alpha Z_1Z_2 \dots Z_n)} M_k
$$
where $M_k$ is the last term of the sequence $M =(M_i)_{i=0, 1, \dots, k}\in  \mathcal M(\alpha Z_1Z_2 \dots Z_n)$. 

As a consequence, if we use the same notation for the monomial $X^{n-u}Y^u=X \dots XY \dots Y \in \A\langle X,Y\rangle_n$ and for its image $X^{n-u}Y^u \in \A[X,Y]^q_n$ under the projection $\pi \colon \A\langle X,Y\rangle_n \to \A[X,Y]^q_n$, 
$$
  \rho_n(A_1A_2 \dots A_k) ( X^{n-u}Y^u ) = \sum_{M \in \mathcal M(X^{n-u}Y^u)} \pi(M_k).
$$

Finally, let $\mathcal M'(X^{n-u}Y^u)$ be the set of monomial sequences $M \in \mathcal M(X^{n-u}Y^u)$ whose contribution  $\pi(M_k) $ belongs to $ \A X^{n-u}Y^u$. For such a monomial sequence $M =(M_i)_{i=0, 1, \dots, k}$, let $\alpha(M_k)\in \A$ be the coefficient such that $\pi(M_k) = \alpha(M_k) X^{n-u}Y^u$. We can then compute the trace of $\rho_n(A_1A_2 \dots A_k)$ by using  the basis $\{ X^{n-u}Y^u; u=0,1,\dots,n\}$ for $ \C[X,Y]^q_n$ and $ \A[X,Y]^q_n = \A \otimes \C[X,Y]^q_n$, which gives
\begin{equation}
\label{eqn:TraceDimNgeneric}
\Tr   \rho_n(A_1A_2 \dots A_k) = \sum _{u=0}^n\  \sum_{M \in \mathcal M'(X^{n-u}Y^u)}  \alpha(M_k). 
\end{equation}

Similarly
\begin{equation}
\label{eqn:TraceDimN-2generic}
\Tr   \rho_{n-2}(A_1A_2 \dots A_k) = \sum _{v=0}^{n-2}\  \sum_{M \in \mathcal M'(X^{n-v-2}Y^v)}  \alpha(M_k). 
\end{equation}

The expression (\ref{eqn:TraceDimNgeneric}) for $\Tr   \rho_n(A_1A_2 \dots A_k)$ was obtained by using the basis $\{ X^{n-u}Y^u; u=0,1,\dots,n\}$ for $ \A[X,Y]^q_n$. For comparison with (\ref{eqn:TraceDimN-2generic}), it is more convenient to use the other basis $\{X^n, Y^n\} \cup \{ X^{n-v-2}Y^{v}XY; v=0,1,\dots,n-2\}$ for $ \A[X,Y]^q_n$. This gives the expression 
\begin{align}
\label{eqn:TraceDimNgenericBis}
\Tr   \rho_n(A_1A_2 \dots A_k) &=  \sum_{M' \in \mathcal M'(X^n)}  \alpha(M_k') +  \sum_{M' \in \mathcal M'(Y^n)}  \alpha(M_k') \\ \notag
&  \qquad\qquad\qquad +\sum _{v=0}^{n-2}\quad \sum_{M' \in \mathcal M'(X^{n-v-2}Y^{v}XY)}  \beta(M_k')
\end{align}
where, for $M' =(M_i')_{i=0, 1, \dots, k} \in \mathcal M'(X^{n-v-2}Y^{v}XY)$, the coefficient $\beta(M_k') \in \A$ is defined by the property that $\pi(M_k') = \beta(M_k') X^{n-v-2}Y^{v}XY$.

Note that, because of the commutativity and $q$--commutativity properties of the quantities $a_i$, $b_i$, $X$, $Y$, all terms $\alpha(M_k)$ and $\beta(M_k')$ in (\ref{eqn:TraceDimN-2generic}--\ref{eqn:TraceDimNgenericBis})  are positive monomials of the form $+ q^\xi \prod_{i=1}^k a_i^{\alpha_i} b_i^{\beta_i} $ with $\xi$, $\alpha_i$, $\beta_i \in \Z$ and $\beta_i\geq 0$. .

We now compare (\ref{eqn:TraceDimN-2generic}) and (\ref{eqn:TraceDimNgenericBis}).  Every monomial sequence  $M =(M_i)_{i=0, 1, \dots, k} \in \mathcal M'(X^{n-v-2}Y^v)$ gives rise to a monomial sequence $M'  \in \mathcal M'(X^{n-v-2}Y^vXY)$ defined by the property that $M_i' = M_iXY$ for every $i$. Indeed, rewriting
$$
M_i' = M_i (a_iX)(a_i^{-1}Y)
$$
shows that $M' =(M_i')_{i=0, 1, \dots, k} $ really satisfies the inductive property defining $\mathcal M'(X^{n-v-2}Y^vXY)$. In addition, when $M'  \in \mathcal M'(X^{n-v-2}Y^vXY)$ is thus associated to $M \in \mathcal M'(X^{n-v-2}Y^v)$, 
$$
\pi(M_k') = \pi(M_kXY) =\pi(M_k)\pi(X)\pi(Y) = \alpha(M_k) X^{n-v-2}Y^vXY 
$$
 so that $\beta(M_k')= \alpha(M_k)$.

 Therefore, when computing 
 $$
T_n \big( \Tr A_1A_2\dots A_k \big ) = \Tr \rho_n ( A_1A_2\dots A_k ) - \Tr \rho_{n-2} ( A_1A_2\dots A_k ),
$$
every monomial $\alpha(M_k)$ occurring in (\ref{eqn:TraceDimN-2generic}) cancels out with a monomial  $\beta(M_k')$ of (\ref{eqn:TraceDimNgenericBis}). It follows that $T_n \big( \Tr A_1A_2\dots A_k \big ) $ is the sum of the remaining coefficients $\alpha(M_k)$ and $\beta(M_k')$ of (\ref{eqn:TraceDimNgenericBis}). We already observed that these monomials are positive, which concludes the proof of Proposition~\ref{prop:TnPositiveSum}. 
\end{proof}
 
 Proposition~\ref{prop:TnPositiveSum} enables us to  precisely determine the number of monomials in $T_n \big( \Tr A_1A_2\dots A_k \big ) $ under the hypothesis that there are no extraneous simplifications. This means that $q$ is transcendental and, since we can always assume that the algebra $\A$ is generated by the entries $a_i$, $b_i$ of the matrices $A_i$, that $\A$ is the algebra defined by the generators $a_i^{\pm 1}$, $b_i$ and by the relations that $b_ia_i = qa_i b_i$ and that $a_i$, $b_i$ commute with $a_j$, $b_j$ whenever $i \neq j$. In other words, $\A$ is the algebra $\bigotimes_{i=1}^k \C[a_i^{\pm1}, b_i]^q$. 
 
 In this case, every element of $\A=\bigotimes_{i=1}^k \C[a_i^{\pm1}, b_i]^q$ has a unique decomposition as a sum of monomials $\xi\prod_{i=1}^k a_i^{\alpha_i} b_i^{\beta_i} $ with  $\xi\in \C$, $\alpha_i\in \Z$, $\beta_i \in \Z$ and $\beta_i\geq 0$. 
  
 \begin{prop}
 \label{prop:TnExactNumberMonomials}
Suppose  that $q$ is transcendental, and that $\A=\bigotimes_{i=1}^k \C[a_i^{\pm1}, b_i]^q$. Let triangular matrices $A_i=  \big( 
\begin{smallmatrix}
 a_i & b_i\\ 0&a_i^{-1}
\end{smallmatrix}\big)$ or 
$\big(\begin{smallmatrix}
 a_i & 0\\ b_i&a_i^{-1}
\end{smallmatrix}
 \big) \in \SLq(\A)$ be given for $i=1$, $2$, \dots, $k$, and consider the positive integer $t_0= \Tr A_1^{(0)}A_2^{(0)}\dots A_k^{(0)}$  where $A_i^{(0)} = \big( 
\begin{smallmatrix}
 1& 1\\ 0&1
\end{smallmatrix}\big)$ if $A_i=  \big( 
\begin{smallmatrix}
 a_i & b_i\\ 0&a_i^{-1}
\end{smallmatrix}\big)$ and $A_i^{(0)}=\big(\begin{smallmatrix}
 1 & 0\\ 1&1
\end{smallmatrix}
 \big)$ if $A_i=\big(\begin{smallmatrix}
 a_i & 0\\ b_i&a_i^{-1}
\end{smallmatrix}
 \big)$. 
Then,  for every $n$, $T_n \big( \Tr A_1A_2\dots A_k \big ) $ is the sum of exactly
 $$\textstyle
T_n(t_0) =\left( \frac{t_0+\sqrt{t_0^2-4}}2 \right)^n + \left( \frac{t_0-\sqrt{t_0^2-4}}2 \right)^n
$$
positive monomials of the form $+ q^\xi \prod_{i=1}^k a_i^{\alpha_i} b_i^{\beta_i} $ with $\xi$, $\alpha_i$, $\beta_i \in \Z$ and $\beta_i\geq0$.
\end{prop}

\begin{proof}
 We already proved in Proposition~\ref{prop:TnPositiveSum} that $T_n \big( \Tr A_1A_2\dots A_k \big ) $ is a sum of monomials of the type indicated. The only issue  is to count their number. 
 
 Because of the positive signs, the number of these monomials can be computed by letting $q$ and the $a_i$, $b_i$ tend to 1. Under this limiting process, $\Tr A_1A_2\dots A_k $ approaches $t_0$, and the number of monomials in the expansion for $T_n \big( \Tr A_1A_2\dots A_k \big ) $ is therefore  equal to~$T_n(t_0)$. 
 
 The formula $T_n(t_0) =\Big( \frac{t_0+\sqrt{t_0^2-4}}2 \Big)^n + \Big( \frac{t_0-\sqrt{t_0^2-4}}2 \Big)^n$ is provided by Lemma~\ref{lem:ChebyshevClosedForm}. 
\end{proof}

\bibliographystyle{amsalpha}
 \bibliography{Cancellations}

\end{document}